\documentclass[]{imsart}

\RequirePackage{amsthm,amsmath,amsfonts,amssymb}
\RequirePackage[numbers]{natbib}
\RequirePackage{xcolor}
\usepackage{xfrac}
\usepackage{hyperref}

\startlocaldefs
\theoremstyle{plain}
\newtheorem{theorem}{Theorem}
\newtheorem{lemma}{Lemma}
\newtheorem{proposition}{Proposition}

\theoremstyle{remark}
\newtheorem{remark}{Remark}
\newtheorem{definition}{Definition}
\newtheorem{example}{Example}

\DeclareSymbolFont{bbold}{U}{bbold}{m}{n}
\DeclareSymbolFontAlphabet{\mathbbold}{bbold}
\newcommand{\ind}{\mathbbold{1}}

\newcommand{\R}{\mathbb{R}}
\newcommand{\N}{\mathbb{N}}
\renewcommand{\P}{\mathbb{P}}
\newcommand{\E}{\mathbb{E}}
\newcommand{\F}{\mathcal{F}}
\newcommand{\supp}{\mathrm{supp}}

\newcommand{\weak}{\tau_{w}}
\newcommand{\weakp}{\tau_{w}^p}

\newcommand{\kelbow}{k^{\textnormal{elb}}}
\newcommand{\clustermedoid}{C^{\textnormal{med}}}
\newcommand{\clusterelbow}{C^{\textnormal{elb}}}

\newcommand{\closed}{\mathrm{C}}
\newcommand{\cbdd}{\mathrm{K}}
\newcommand{\kurouter}{\mathrm{Ls}}
\newcommand{\kurinner}{\mathrm{Li}}
\newcommand{\kurlimit}{\mathrm{Lt}}
\newcommand{\dvecHaus}{\vec{d}_{\mathrm{H}}}
\newcommand{\dHaus}{d_{\mathrm{H}}}
\newcommand{\hausdorfftopo}{\tau_{H}}
\renewcommand{\emptyset}{\varnothing}
\newcommand{\diff}{\,\textnormal{d}}

\hypersetup{
	colorlinks=true,       
	linkcolor=blue,          
	citecolor=blue,        
	filecolor=blue,      
	urlcolor=blue           
}

\endlocaldefs

\begin{document}

\begin{frontmatter}
\title{Asymptotic Theory of Geometric and Adaptive $k$-Means Clustering}

\begin{aug}
\author[A]{\fnms{Adam Quinn} \snm{Jaffe}\ead[label=e1]{a.q.jaffe@columbia.edu}},
\address[A]{Department of Statistics, Columbia University, \printead{e1}}
\end{aug}

\begin{abstract}
	We revisit Pollard's classical result on consistency for $k$-means clustering in Euclidean space, with a focus on extensions in two directions:
	first, to problems where the data may come from interesting geometric settings (e.g., Riemannian manifolds, reflexive Banach spaces, or the Wasserstein space);
	second, to problems where some parameters are chosen adaptively from the data (e.g., $k$-medoids or elbow-method $k$-means).
	Towards this end, we provide a general theory which shows that all clustering procedures described above are strongly consistent.
	In fact, our method of proof allows us to derive many asymptotic limit theorems beyond strong consistency.
	We also remove all assumptions about uniqueness of the set of optimal cluster centers.
\end{abstract}



\end{frontmatter}


\setcounter{tocdepth}{1}
\tableofcontents

\section{Introduction}

A fundamental task in unsupervised learning is that of \textit{clustering}, namely, partitioning a set of data into a finite number of groups where elements within a group are similar (and, typically, elements between distinct groups are dissimilar).
Among the most common clustering methods is $k$-means clustering  \cite{MacQueen}:
For data points $Y_1,\ldots Y_n$ in the Euclidean space $\R^m$ and any $k\in\N:=\{1,2,\ldots\}$, the set of \textit{$k$-means cluster centers} is any solution to the set-indexed optimization problem
\begin{equation}\label{eqn:kmeans-emp}
	\begin{cases}
		\text{minimize}\qquad &\frac{1}{n}\sum_{i=1}^{n}\min_{x\in S}|x-Y_i|^2 \\
		\text{subject to}\qquad &S\subseteq \R^m \text{ and } 1\le \#S \le k.
	\end{cases}
\end{equation}
Intuitively speaking, a set of $k$-means cluster centers for these data points is a set of points $S_n$ in $\R^m$ to at least one of which all data are optimally close; the \textit{$k$-means clusters} are then the sets
\begin{equation*}
\{Y_i : |x-Y_i|\le |x'-Y_i| \text{ for all } x'\in S_n\}
\end{equation*}
indexed by $x\in S_n$.
The problem of $k$-means clustering is easily seen to be equivalent to \textit{vector quantization} \cite{AbayaWise, Linder, Laloe, Quant_kmeans} and to \textit{principal points} \cite{Flury1, Flury2, ClusteringFunctional} that have been studied in other disciplines.

The initial theoretical justification for $k$-means is Pollard \cite{Pollard} in which strong consistency is shown  for a wide class of non-parametric models; that is, if $Y_1,Y_2,\ldots$ is an independent, identically-distributed (IID) sequence of random variables in a finite-dimensional Euclidean space $\R^m$, and if a number of technical hypotheses hold, then the sequence of solutions $S_1,S_2,\ldots$ to the optimization problems above has a limit almost surely, in a sense which we will soon make precise.
Moreover, the limit is a solution to the set-indexed optimization problem
\begin{equation}\label{eqn:kmeans-pop}
\begin{cases}
	\text{minimize}\qquad &\int_{\R^m}\min_{x\in S}|x-y|^2\, \diff\mu(y) \\
	\text{subject to}\qquad &S\subseteq \R^m \text{ and }1\le \#S \le k.
\end{cases}
\end{equation}
The arguments therein rely on delicate calculations which combine uniform laws of large numbers with some recursive structure that relates the $k$-means cluster centers to the $(k-1)$-means cluster centers. 

While this strong consistency result and its subsequent developments \cite{Parna1,Parna2,Parna3,Lember} are applicable in some settings, they are lacking in a few important ways.
In this work, we focus on extending such results in two directions.
First, we want to provide guarantees for $k$-means clustering applied to data that live in geometric settings other than $\R^m$.
Second, we want to provide guarantees for variants of $k$-means clustering in which some parameters are chosen adaptively from the data.
At the same, we aim to develop an asymptotic theory which establishes various limit theorems beyond strong consistency.


The remainder of this introduction is divided into further subsections.
In Subsection~\ref{subsec:variants} we define several geometric and adaptive clustering procedures which are variants of $k$-means clustering that are of interest in statistics and machine learning.
In Subsection~\ref{subsec:results} we precisely state our main results.
In Subsection~\ref{subsec:related-lit} we review some literature related to the present work.
The remainder of the paper is dedicated to the proofs of our main results, although some proofs are deferred to the supplementary material 

Finally, we make a remark on a particular piece of terminology used throughout the paper.
While we will always discuss \textit{clustering procedures}, some authors prefer to discuss \textit{clustering algorithms}.
To make things concrete for this paper, a \textit{procedure} will always refer to the setting in which one has oracle access to the solution set of each optimization problem.
In contrast, an \textit{algorithm} always refers to a particular method of computing such optimizers (or near-optimizers).
With a few exceptions, we will only consider procedures in this work, although the algorithmic questions related to clustering are themselves highly non-trivial.

\subsection{Variants of $k$-means clustering}\label{subsec:variants}

In addition to the classical $k$-means clustering problem introduced above, we now describe several variants which have become important in various statistical applications.
The main results of the paper, stated in the next subsection, will allow us to treat all of these examples simultaneously.

\subsubsection{Geometric variants}\label{subsubsec:geometric}

Since many modern statistical problems involve data that do not simply live in Euclidean space $\R^m$, there has been growing interest in applications of $k$-means clustering to more interesting geometric setttings.
Presently, we give an overview of some of these geometries and the relevant applications.

\medskip

\noindent
\textbf{Riemannian Manifolds.} Applications in computer vision \cite{SrivastavaRiemannianClustering}, bioinformatics \cite{ProteinClustering}, and air traffic control \cite{LeBrigantRiemannian} have introduced the following variant of $k$-means clustering, where $\mathcal{M}$ denotes a Riemannian manifold, $d$ denotes its geodesic distance, and $q_1,\ldots, q_n$ denote some arbitrary data points in $\mathcal{M}$:
\begin{equation*}
	\begin{cases}
		\text{minimize}\qquad &\frac{1}{n}\sum_{i=1}^{n}\min_{p\in S}d^2(p,q_i) \\
		\text{subject to}\qquad &S\subseteq \mathcal{M} \text{ and } 1\le \#S \le k.
	\end{cases}
\end{equation*}
Asymptotic consistency for this so-called \textit{Riemannian $k$-means clustering problem} follows from general theory for Heine-Borel metric spaces, that is metric spaces whose closed balls are compact; that consistency should hold for such spaces was already remarked by Pollard in \cite{Pollard}, and subsequently proven in \cite{Parna1, Parna2}.
See also \cite{ManifoldClustering, GeomStats} for related computational concerns.

\medskip

\noindent
\textbf{Banach Spaces.} In applications of functional data analysis, for example to meteorology and finance \cite{ClusteringFunctional}, one often encounters the following problem:
for a reflexive Banach space ${(\mathcal{B},\|\cdot\|)}$ and some functions $\phi_1,\ldots, \phi_n$ in $\mathcal{B}$, solve
\begin{equation*}
	\begin{cases}
		\text{minimize}\qquad &\frac{1}{n}\sum_{i=1}^{n}\min_{\psi\in S}\|\psi-\phi_i\|^2 \\
		\text{subject to}\qquad &S\subseteq \mathcal{B} \text{ and } 1\le \#S \le k.
	\end{cases}
\end{equation*}
For this problem of clustering functional data, many general asymptotic consistency results are known, including with respect to weak convergence in reflexive Banach spaces \cite{Thorpe,Parna3} and with respect to strong convergence in uniformly convex Banach spaces \cite{Lember}.

\medskip

\noindent
\textbf{Wasserstein Space.} Recent applications of clustering distribution data, for example in demography \cite{delBarrioClustering} and remote sensing \cite{PapayiannisClustering}, have introduced the following clustering problem, where $\mathcal{P}_2(\R^m)$ denotes the 2-Wasserstein space of probability measures and $W_2$ the 2-Wasserstein metric on $\mathcal{P}_2(\R^m)$:
if one has data $\nu_1,\ldots, \nu_n$ which are probability measures on $\R^m$, then it is natural to try to solve
\begin{equation*}
	\begin{cases}
		\text{minimize}\qquad &\frac{1}{n}\sum_{i=1}^{n}\min_{\mu\in S}W_2^2(\mu,\nu_i) \\
		\text{subject to}\qquad &S\subseteq \mathcal{P}_2(\R^m)\text{ and } 1\le \#S \le k,
	\end{cases}
\end{equation*}
which is called the \textit{Wasserstein $k$-means} or \textit{Wasserstein $k$-barycenters} problem.
As far as we are aware, the literature contains no results regarding asymptotic theory for this problem as $n\to\infty$.
Rather, the focus of existing works is computational, and aims to develop efficient implementation or approximation of the optimal cluster centers \cite{ZhuangWasserstein, VerdinelliWasserman}.

\medskip

As we will later see, it natural to view all of these problems as $k$-means clustering in a metric space, which must be assumed to have some mild additional structure.
This will allow us to also treat some more niche examples, like the tropical projective metric with applications to phylogeny \cite{TropTorusKMeans}, the partial matching metric with applications to shape matching \cite{PersHomKMeans}, and the rotationally-invariant Wasserstein metric with applications to cryogenic electronic microcroscopy (cryo-EM) \cite{WassersteinTomography}.

\subsubsection{Adaptive variants}\label{subsubsec:adaptive}

In most practical applications of $k$-means clustering, some elements of the problem are not fixed but rather they are chosen as a function of the data.
Presently, we overview two of these adaptive clustering procedures and some motivations.
Throughout this section, let $(\mathcal{X},d)$ denote a general metric space and $y_1,\ldots, y_n$ some points in $\mathcal{X}$.

\medskip

\noindent
\textbf{$\boldsymbol{k}$-medoids.} In many applications of clustering it is desirable to have cluster centers which are themselves data points rather than just abstract points in space.
(See \cite[Chapter~2]{KR_book}, \cite[Section~14.3.10]{ESL}, or \cite[p. 208-209]{Molar}.).
Towards, this end \cite{Kaufman}, it is natural to fix $k\in \N$ and to consider
\begin{equation}\label{eqn:kp-medoids}
\begin{cases}
\text{minimize}\qquad &\frac{1}{n}\sum_{i=1}^{n}\min_{x\in S}d^2(x,y_i) \\
\text{subject to}\qquad &S\subseteq \{y_1,\ldots y_n\} \text{ and } 1\le\#S \le k,
\end{cases}
\end{equation}
called the \textit{$k$-medoids clustering problem.}
Recent work \cite{JiangAriasCastro} has shown the consistency of $k$-medoids under some assumptions, but most existing work in this domain is computational \cite{ParkJunMedoids, BanditPAM, LinearMedoids, UltraFastMedoids}.

\medskip

\noindent
\textbf{$\boldsymbol{k}$-means with elbow method.} One of the fundamental problems of $k$-means clustering is that the practitioner must somehow decide the number of clusters $k$ to choose; a common approach, mentioned in nearly every introductory textbook on machine learning (see \cite[Section~14.3.11]{ESL}, \cite[p. 247-248]{GeronML}, or \cite[Section 7.9]{Alpaydin}) is to select $k$ via the so-called \textit{elbow method} which chooses the value of $k$ for which the added model complexity experiences maximally dimishing returns.
More precisely, for each $k\in\N$ we set
\begin{equation*}
	m_{k} := \inf_{\substack{S\subseteq \mathcal{X} \\ 1\le\#S\le k}}\frac{1}{n}\sum_{i=1}^{n}\min_{x\in S}d^2(x,y_i),
\end{equation*}
which is the minimal objective achievable by any set of cluster centers.
Then define the discrete second derivative for $k\ge 2$ via $\Delta^2m_{k}  :=m_{k+1} + m_{k-1}  - 2m_{k}$, and set
\begin{equation*}
	\kelbow :=\min\{\arg\max\{\Delta^2m_{k}: k \in\N, k\ge 2\}\}.
\end{equation*}
(Note that the restriction $k\ge 2$ can be understood as taking the convention that $m_{0} = \infty$, which equivalently means that $k=1$ will never be selected as the number of clusters.)
Now we consider
\begin{equation}\label{eqn:adapt-kp-means}
	\begin{cases}
		\text{minimize}\qquad &\frac{1}{n}\sum_{i=1}^{n}\min_{x\in S}d^2(x,y_i) \\
		\text{subject to}\qquad &S\subseteq \mathcal{X} \text{ and } 1\le \#S \le \kelbow.
	\end{cases}
\end{equation}
This is a naive formalism of the well-known procedure which selects the ``elbow'' from the plot of the sequence $\{m_k\}_{k\in\N}$.
Despite the ubiquity of this adaptive method of selecting $k$, we are not aware of any literature providing asymptotic theory for the resulting adaptive clustering procedure.

\subsection{Statement of results}\label{subsec:results}

To begin, let us describe the clustering procedures of interest.
Throughout, we fix a metric space $(\mathcal{X},d)$ and a real number $p\ge 1$.
For any distribution $\mu\in \mathcal{P}(\mathcal{X})$ (here, $\mathcal{P}(\mathcal{X})$ denotes the space of Borel probability measures on $(\mathcal{X},d)$) with $\int_{\mathcal{X}}d^p(x,y)\, \diff\mu(y) < \infty$ for all $x\in \mathcal{X}$, any integer $k\in\N$,  and any subset $R\subseteq \mathcal{X}$, we consider the problem of choosing the best set of cluster centers $S\subseteq R$ with $1\le \#S\le k$, as quantified by the loss
\begin{equation*}
	\int_{\mathcal{X}}\min_{x\in S}d^p(x,y)\, \diff\mu(y).
\end{equation*}
More precisely, we write $C_{p}(\mu,k,R)$ for the set of all $S\subseteq R$ with $1\le \#S\le k$ satisfying
\begin{equation*}
	\int_{\mathcal{X}}\min_{x\in S}d^p(x,y)\, \diff\mu(y) = \inf_{\substack{S'\subseteq R \\1\le \#S'\le k}} \int_{\mathcal{X}}\min_{x\in S'}d^p(x,y)\, \diff\mu(y) =:m_{k,p}(\mu).
\end{equation*}
Note that $C_{p}(\mu,k,R)$ is, in general, a set of subsets of $\mathcal{P}(\mathcal{X})$.
We call $C_p(\mu,k,R)$ the \textit{set of sets of $R$-restricted $(k,p)$-means cluster centers.}

First we describe the geometric conditions required on the metric space $(\mathcal{X},d)$, which have been studied in detail by the author in the companion paper \cite{JaffeInfDim}.
Rougly speaking, our results will require that $(\mathcal{X},d)$ admits a suitably powerful notion of weak convergence, to be described below.
Importantly, in \cite[Section~2.2]{JaffeInfDim} it is shown that Riemannian manifolds, uniformly convex Banach spaces, and the Wasserstein space all satisfy this property.
Further, results contained in \cite[Section~2.4]{JaffeInfDim} show that the existence of a suitable weak convergence is also satisfied by the additional examples of the tropical projective metric, the partial matching metric, and the the rotationally-invariant invariant Wasserstein metric.
Thus, our results apply to all of the metric spaces of Subsubsection~\ref{subsubsec:geometric}.

To state this geometric condition precisely, let us say that a metric space $(\mathcal{X},d)$ \textit{admits a weak convergence} if their exists a Hausdorff topology $\tau$ on $\mathcal{X}$ satisfying the following conditions:
\begin{enumerate}
	\item[(W1)] If $\{x_n\}_{n\in\N}$ and $y\in \mathcal{X}$ have $\sup_{n\in\N}d(x_n,y)<\infty$, then there exists a subsequence $\{n_j\}_{j\in\N}$ and a point $x\in \mathcal{P}(\mathcal{X})$ with $x_{n_j}\to x$ in $\tau$.
	\item[(W2)] If $\{x_n\}_{n\in\N}$, and $x$ in $\mathcal{X}$ have $x_n\to x$ in $\tau$, then we have $d(x,y)\le \liminf_{n\to\infty}d(x_n,y)$ for all $y\in \mathcal{X}$.
	\item[(W3)] If $\{x_n\}_{n\in\N}$ and $x$ in $\mathcal{X}$ have $x_n\to x$ in $\tau$ and if there exists some $y\in \mathcal{X}$ with $d(x_n,y)\to d(x,y)$, then $x_n\to x$ in $d$.
\end{enumerate}
To see that the metric spaces above indeed admit a weak convergence, we note that one can take $\tau$ to be the metric topology in the case of Riemannian manifolds, one can take $\tau$ to be the weak topology in the case of uniformly convex Banach spaces, and one can take $\tau$ to be the topology of weak convergence of measures in the case of the Wasserstein space; for the full analytic details, we direct the reader to \cite{JaffeInfDim}.

\begin{remark}
	A potential cause for confusion throughout the paper is that there will be, occasionally, two topologies at play: the topology generated by the metric $d$ and the topology $\tau$.
	In general, most topological statements refer to the topology generated by $d$, but $\tau$ must occasionally be used.
	To remedy this, we adopt the following conventions:
	when no explicit reference is made, it will be understood that the topology generated by $d$ is the relevant topology; when there is risk of ambiguity, we will explicitly identify which topology is relevant.
	Also, by a slight abuse of notation, we use $d$ for both the metric and also for the topology generated by this metric.
\end{remark}

For example, $\mathcal{P}(\mathcal{X})$ refers to the space of all Borel probability measures on $\mathcal{X}$, where the Borel $\sigma$-algebra is generated by the topology of $d$.
Moreover, for $\mu\in\mathcal{P}(\mathcal{X})$, we write $\supp(\mu):=\{x\in \mathcal{X}: \textnormal{if $U$ is $d$-open and $x\in U$, then $\mu(U)>0$}\}$ and $\supp^{\tau}(\mu):=\{x\in \mathcal{X}: \textnormal{if $U$ is $\tau$-open and $x\in U$, then $\mu(U)>0$}\}$ to denote the topological support of $\mu$ with respect to $d$ and $\tau$, respectively.
Note, importantly, that if $\mu=\frac{1}{n}\sum_{i=1}^{n}\delta_{y_i}$ for some $n\in\N$ and $y_1,\ldots, y_n\in X$, then $\supp(\mu) = \supp^{\tau}(\mu) = \{y_1,\ldots, y_n\}$.

Second, we describe the conditions on the adaptive clustering procedures that we will require.
To do this, we observe that $C_p(\mu,k,R)$ indeed unifies all of the clustering procedures of Subsection~\ref{subsubsec:adaptive}:
writing $\bar \mu_n := \frac{1}{n}\sum_{i=1}^{n}\delta_{y_i}$ for the empirical measure of the first $n\in\N$ samples, we see that $C_{k,2}(\bar \mu_n) := C_{2}(\bar \mu_n,k,\mathcal{X})$ is exactly the set of sets of $k$-means cluster centers, $\clustermedoid_{k,2}(\bar \mu_n) := C_2(\bar \mu_n,k,\{y_1,\ldots, y_n\})$ is exactly the set of sets of $k$-medoids cluster centers, and $\clusterelbow_{2}(\bar \mu_n):=C_2(\bar \mu_n,\kelbow(\bar \mu_n),\mathcal{X})$ is exactly the set of sets of $k$-means cluster centers when $k$ is chosen adaptively according to the elbow method.
So, it is natural that our results will, roughly speaking, show that we have $C_p(\mu_n,k_n,R_n)\to C_p(\mu,k,R)$ provided that we have $\mu_n\to \mu$, $k_n\to k$, and $R_n\to R$ in suitable senses.

To state this convergence precisely, we need to introduce some notions of convergence of sets.
First, for non-empty compact subsets $A,A'\subseteq \mathcal{X}$, we define their \textit{Hausdorff distance} to be
\begin{equation*}
	\dHaus(A,A') := \inf\left\{r \ge 0: A\subseteq \bigcup_{x\in A'}\bar B^d_{r}(x) \text{ and } A'\subseteq \bigcup_{x\in A}\bar B^d_{r}(x) \right\},
\end{equation*}
where $\bar B^d_r(x):=\{y\in \mathcal{X}: d(x,y)\le r \}$ is the $d$-closed ball of radius $r\ge 0$ around $x\in \mathcal{X}$.
Convergence with respect to the Hausdorff metric is rather strong, so we also need the weaker notion of \textit{Kuratowski convergence}; for $\{R_n\}_{n\in\N}$ and $R$ any closed subsets of $\mathcal{X}$ (with respect to the topology generated by $d$), we define:
\begin{align*}
	\underset{n\to\infty}{\kurouter}R_n &:= \left\{x\in \mathcal{X}:{\begin{matrix}{\mbox{for all open neighborhoods }}U{\mbox{ of }}x,\\ U\cap R_{n}\neq \emptyset \mbox{ for  infinitely many } n\in\N\end{matrix}}\right\} \\
	\underset{n\to\infty}{\kurinner}R_n &:= \left\{x\in \mathcal{X}:{\begin{matrix}{\mbox{for all open neighborhoods }}U{\mbox{ of }}x,\\U\cap R_{n}\neq \emptyset {\mbox{ for large enough }}n \in \N\end{matrix}}\right\},
\end{align*}
called the \textit{Kuratowski upper limit} and the \textit{Kuratowski lower limit}, respectively.
We always have $\kurinner_{n\in\N}R_n \subseteq \kurouter_{n\in\N}R_n$, and we write $\kurlimit_{n\in\N}R_n = R$ to mean that we in fact have $\kurouter_{n\in\N}R_n \subseteq R\subseteq \kurinner_{n\in\N}R_n$.
Similarly, we write $\kurinner_{n\in\N}^{\tau}R_n$ and $\kurouter_{n\in\N}^{\tau}R_n$ when $\{R_n\}_{n\in\N}$ are $\tau$-closed sets and the neighborhoods above are taken to be $\tau$-open. 
We will also need some notions of random sets; their precise definitions can be found in Subsection~\ref{subsec:subsets}.

Now we can give the primary statistical setting of interest.
For any $\mu\in \mathcal{P}(\mathcal{X})$, we let $(\Omega,\F,\P)$ be a complete probability space on which is defined an  IID sequence $Y_1,Y_2,\ldots$ of $\mathcal{X}$-valued random variables with common distribution $\mu$, and write $\bar \mu_n := \frac{1}{n}\sum_{i=1}^{n}\delta_{Y_i}$ for the empirical probability measure of the first $n\in\N$ samples.
Then we have the following:

\begin{theorem}\label{thm:Haus-consistency}
	Suppose that $(\mathcal{X},d)$ is a separable metric space admitting a weak convergence, and that $\mu\in\mathcal{P}(\mathcal{X})$ satsfies $\int_{\mathcal{X}}d^p(x,y)\, \diff\mu(y) < \infty$ for all $x\in \mathcal{X}$.
	Then, for any random $\tau$-closed subsets $\{R_n\}_{n\in\N}$ and $R$ of $\mathcal{X}$ and any random positive integers $\{k_n\}_{n\in\N}$ and $k$, we have
	\begin{equation*}
		\max_{S_n\in C_{p}(\bar \mu_n,k_n,R_n)}\min_{S\in C_{p}(\mu,k,R)}\dHaus(S_n,S)\to 0
	\end{equation*}
	almost surely on
	\begin{align*}
		&\{k\le \#(R\cap\supp(\mu))\} \cap \{k_n \neq k \text{ finitely often}\}\\
		&\qquad\cap\{\kurouter^{\tau}_{n\in\N}R_n\subseteq R\subseteq \kurinner_{n\in\N}R_n\}.
	\end{align*}
\end{theorem}

Some remarks about this result are due.
First, we regard Theorem~\ref{thm:Haus-consistency} as a guarantee of ``no false positives'', in an asymptotic sense, since it establishes that any set of empirical $R_n$-restricted $(k_n,p)$-means cluster centers must be close to some set of population $R$-restricted $(k,p)$-means cluster centers, as $n\to\infty$.
Second, we emphasize that the convergence above holds with respect to $d$ and not with respect to $\tau$; in this sense, the existence of a weak convergence is merely a geometric property of the space $(\mathcal{X},d)$, but it does not feature in the conclusion.
Third, we note (see Proposition~\ref{prop:cpt-operator}) that one can derive a similar but weaker conlusion under the weaker assumption that $(\mathcal{X},d)$ is a separable metric space admitting a topology $\tau$ satisfying only (W1) and (W2); this applies, for example, to the case of reflexive Banach spaces and allows us to recover and strengthen the results of \cite{Thorpe}.

The majority of the body of the paper is spent proving Theorem~\ref{thm:Haus-consistency}, so let us show that it indeed implies strong consistency for the clustering procedures introduced in Subsection~\ref{subsec:variants}.
As we have already discussed, all of the metric spaces of Subsection~\ref{subsubsec:geometric} admit a weak convergence.
To see that the adaptive variants described in Subsection~\ref{subsubsec:adaptive} are covered by this result, we consider the following.
Suppose $\int_{\mathcal{X}}d^p(x,y)\, \diff\mu(y) < \infty$ holds for all $x\in \mathcal{X}$, as well as $\#\supp(\mu) = \infty$ for simplicity, which certainly holds in most applications.
\begin{itemize}
\item For $(k,p)$-means, simply take $k\in\N$ and $k_n = k$ and $R_n = R = \mathcal{X}$ for $n\in\N$.
\item For $(k,p)$-medoids, take $k\in\N$ and $R = \supp^{\tau}(\mu)$, as well as $k_n = k$ and $R_n = \{Y_1,\ldots, Y_n\}$ for all $n\in\N$.
We may conclude that $(k,p)$-medoids is strongly consistent, provided that we have $\kurouter_{n\in\N}^{\tau}\supp^{\tau}(\bar \mu_n)\subseteq \supp^{\tau}(\mu)\subseteq \kurinner_{n\in\N}\supp^{\tau}(\bar \mu_n)$ almost surely; we prove this easily in Lemma~\ref{lem:supp-converge} below.
\item For elbow-method $(k,p)$-means, take $k = \kelbow_p(\mu)$ and $R = \mathcal{X}$, as well as $k_n = \kelbow_p(\bar \mu_n)$ and $R_n = \mathcal{X}$ for all $n\in\N$.
We may conclude that elbow-method $(k,p)$-means is strongly consistent, provided that we have $\kelbow_p(\bar \mu_n)\to \kelbow_p(\mu)$ almost surely; we prove this in Proposition~\ref{prop:elbow-k-cts} below, under the assumption that $\#\arg\max\{\Delta^2m_{k,p}(\mu): k\in\N, k\ge 2\} =1$, which simply means that $\mu$ has a well-defined number of clusters in the sense of the elbow method.
\end{itemize}
Notice that, in all of these settings, strong consistency follow immediately from Theorem~\ref{thm:Haus-consistency} plus some auxiliary results.

Next, we discuss our method of proof, which is novel and has many advantages over existing methods.
To state it, write $\closed(\mathcal{X})$ for the collection of $d$-closed subsets of $\mathcal{X}$ and write $\cbdd(\mathcal{X})$ for the collection of non-empty $d$-compact subsets of $\mathcal{X}$.
It turns out that $\dHaus$ makes $\cbdd(\mathcal{X})$ into a metric space itself, so we can consider the space $\cbdd(\cbdd(\mathcal{X}),\dHaus)$; this is the collection of non-empty $\dHaus$-compact sets of non-empty $d$-compact subsets of $\mathcal{X}$.
Our proof is to show directly that these spaces can be ``topologized'' such that the deterministic clustering map
\begin{equation*}
	C_p:\mathcal{P}(\mathcal{X})\times \N\times\closed(\mathcal{X})\to \cbdd(\cbdd(\mathcal{X}),\dHaus)
\end{equation*}
is ``continuous'' in a suitable sense.
These fully deterministic statements mean that certain ``convergent sequences'' in the domain are mapped to ``convergent sequences'' in the codomain, although the particular notions of convergence we are interested in may be too degenerate to correspond to convergence in a bona fide topology.
This ``continuity'' is proven in our main technical result (Proposition~\ref{prop:cpt-operator}) which readily implies Theorem~\ref{thm:Haus-consistency}.
 
By this continuity perspective, we see that the strong consistency for clustering procedures is merely descended from an analogous strong law of large numbers for empirical measures of IID samples.
Moreover, through this lens it is easy to see that there is nothing special about our apparent focus on strong limit theorems; we can actually descend many types of limit theorems for empirical measures down to analogous limit theorems for clustering procedures.
We conclude this subsection by describing two applications of this, which we will carefully prove and analyze at the end of the paper; for the sake of simplicity, we focus on the case of non-adaptive $k$ for these resuts.

One important application of this perspective is analyzing strong consistency for clustering procedures when applied to data coming from a more complicated dependency structure, for example from a suitable Markov chain (MC).
Such considerations may arise when data is only assumed to be weakly dependent, or, in Bayesian statistics, when one attempts to apply clustering to a posterior distribution for which samples are computed via Markov chain Monte Carlo (MCMC).
For such settings we have the following result:

\begin{theorem}\label{thm:MC-consistency}
	Suppose that $(\mathcal{X},d)$ is a separable metric space admitting a weak convergence, and that $Y_1,Y_2,\ldots$ is an aperiodic Harris-recurrent Markov chain in $\mathcal{X}$ with stationary distribution $\mu$ satisfying $\int_{\mathcal{X}}d^p(x,y)\, \diff\mu(y) < \infty$ for all $x\in \mathcal{X}$.
	Then, for any $k\le \#\supp(\mu)$, we have
	\begin{equation*}
		\max_{S_n\in C_{k,p}(\bar \mu_n)}\min_{S\in C_{k,p}(\mu)}\dHaus(S_n,S)\to 0
	\end{equation*}
	almost surely.
\end{theorem}

This gives a wide condition guaranteeing strong consistency for $(k,p)$-means when applied to data coming from a MC.
A particular case of interest from the perspective of MCMC is $\mathcal{X} = \R^m$ for $m\in\N$ with $d$ the usual Euclidean metric, and where $\mu$ and $\P(Y_2\in\,\cdot\,|\,Y_1=x)$ for all $x\in \R^m$ have strictly positive densities with respect to the Lebesgue measure; in this case we see that the moment condition $\int_{\R^m}d^p(x,y)\, \diff\mu(y) < \infty$ readily implies strong consistency.

The question of strong consistency for adaptive clustering procedures applied to data coming from a MC is, interestingly, more subtle.
As we detail in Subsection~\ref{subse:results-MCs}, it turns out that a naive application of $(k,p)$-medoids can fail to be strongly consistent, but we are easily able to use the tools of this paper to develop a slight adaptation of $(k,p)$-medoids which is strongly consistent.

Another important application of our perspective is analyzing the quantitative rate of convergence for clustering procedures, for example by means of understanding its large deviations behavior.
For this we return to the setting wherein $\mu\in \mathcal{P}(\mathcal{X})$ is a fixed probability measure and $Y_1,Y_2,\ldots$ are independent, identically-distributed samples of $\mathcal{X}$-valued random variables with common distribution $\mu$.
Then:

\begin{theorem}\label{thm:LDP}
	Suppose $(\mathcal{X},d)$ is a Polish metric space admitting a weak convergence, and that $\int_{\mathcal{X}}\exp(\alpha d^p(x,y))\, \diff\mu(y) < \infty$ holds for all $\alpha > 0$ and all $x\in \mathcal{X}$.
	Then, for all $k\le \#\supp(\mu)$  and  $\varepsilon > 0$, we have
	\begin{equation*}
	\limsup_{n\to\infty}\frac{1}{n}\log \P\left(\sup_{S_n\in C_{k,p}(\bar \mu_n)}\inf_{S\in C_{k,p}(\mu)}\dHaus(S_n,S)\ge \varepsilon\right) \le -c_{k,p}(\mu,\varepsilon),
	\end{equation*}
	for some $c_{k,p}(\mu,\varepsilon) > 0$.
\end{theorem}

While this result is an interesting starting point  for understanding concentration of measure for clustering procedures, there are many natural questions which remain to be answered; we address some of these at the end of Subsection~\ref{subse:results-LDP}.

\subsection{Related literature}\label{subsec:related-lit}

In this subsection we briefly review some research, both classical and modern, which is related to the results in the present paper.

Among all clustering procedures, $k$-means has certainly received the most attention.
A typical reference for the initial theoretical study of $k$-means is usually taken to be \cite{MacQueen}, but it is known \cite{Bock} that the procedure has a rich prehistory under many different names.
The initial theoretical justification for $k$-means (as well as $k$-medians) is Pollard \cite{Pollard} in which a form of strong consistency was first shown.

In most of the recent lines of research on $k$-means clustering, the goal has been to establish consistency in the form of convergence of the empirical objectives to the optimal objective.
In this setting there now exist various forms of non-asymptotic bounds on the excess objective, often in the infinite-dimensional setting of separable Hilbert spaces or separable reflexive Banach spaces \cite{Laloe, BiauDevroyeLugosi, Fefferman, Zhivotovskiy, Parna3}.

A more complicated notion of consistency (which is the focus of our main results) is the convergence of the empirical minimizers to the population minimizers.
Initial results of this form often relied on the assumption that the population distribution admits a unique set of $k$-means cluster centers \cite{Pollard, AbayaWise, CuestaMatran}, but this condition can be very hard to verify (or in fact false) in practice.
Later results were able to dispense with the uniqueness condition \cite{Lember, Parna1, Parna2, Parna3, Thorpe}, sometimes at the expense of proving a slightly weaker notion of convergence.
Similar consistency results are known for $k$-medoids under some assumptions \cite{JiangAriasCastro}.
In context, our result recovers all of the consistency results stated above; in fact, it extends such results in both the geometric and adaptive directions, as we now describe.

On the geometric side, our results apply to variants of $k$-means clustering where the data may live in a very general metric space.
For example, it recovers known results in Heine-Borel metric spaces \cite{Parna1, Parna2} including Riemannian manifolds, and it recovers known results in reflexive Banach spaces \cite{Thorpe, Parna3} and uniformly convex Banach spaces \cite{Lember}.
Importantly, our results also provide new consistency results for $k$-means clustering in the Wasserstein space; while there is an emerging literature studying various theoretical, applied, and computational aspects of $k$-means clustering in the Wasserstein space \cite{delBarrioClustering,PapayiannisClustering, ZhuangWasserstein, VerdinelliWasserman}, our results provide, to the best of our knowledge, a novel asymptotic theory. 
In fact, our results apply to all metric spaces admitting a weak convergence, and many easy-to-verify sufficient conditions for this have been given in author's concurrent work \cite{JaffeInfDim}; we believe this may be of interest in future $k$-means clustering problems.
For example, one can use the results of \cite[Section~2]{JaffeInfDim} to see that the present results also apply to $k$-means clustering in the rotationally-invariant Wasserstein space which has been studied in cryo-EM \cite{WassersteinTomography}.

On the adaptive side, our results apply to variants of $k$-means clustering where parameters may be chosen to depend on the data.
For example, it recovers the main result of \cite{JiangAriasCastro} on consistency for $k$-medoids, and strengthens it by dispensing with the conditions therein regarding local compactness of $(\mathcal{X},d)$ and the bounded support of $\mu$.
Furthermore, our results provide new asymptotic theory for problems in which $k$ is chosen adaptively, for example according to the elbow method; while this method is mentioned in nearly every introductory machine learning textook (see, for example, \cite[Section~14.3.11]{ESL}, \cite[p. 247-248]{GeronML}, or \cite[Section 7.9]{Alpaydin}) we are not aware of any rigorous guarantees that such adaptive choices of $k$ lead to consistent clustering procedures.
We believe that it would be interesting to dedicate future work to studying strong consistency under alternative adaptive choices of $k$, for example the silhouette statistic or the gap statistic \cite{gap}.

Of course, the most significant contribution of this work is that our ``continuity''-based method of proof for these strong consistency results is very robust.
For example, we easily prove Theorem~\ref{thm:MC-consistency} and Theorem~\ref{thm:LDP}, which cannot be easily proved by the standard approaches used by previous authors.
The work of P\"{a}rna \cite{Parna1,Parna2, Parna3} contains traces of this method, since therein it is also observed that, from the point of view of asymptotic theory, the IID structure of the data is primarily useful insofar as it provides almost sure weak convergence of the empirical measures to the population distribution.
The author's own work on Fr\'echet means \cite{EvansJaffeFrechet, JaffeInfDim} is another precedent for this approach, although the present results generalize these significantly.

\medskip

The remainder of the paper is structured as follows.
In Section~\ref{sec:prelim} we review some preliminary topological concepts related to metric spaces, spaces of subsets of a given metric space, and spaces of probability measures on a given metric space.
In Section~\ref{sec:cluster} we study the clustering map itself, and we show that it is, in several suitable senses, ``continuous''.
Finally, we have Section~\ref{sec:results} in which we prove our main probabilistic results.

\section{Preliminaries}\label{sec:prelim}

In this section we develop the basic results which will combine in the next sections to prove various limit theorems for a wide class of clustering procedures.
Unless otherwise stated, we always assume that $(\mathcal{X},d)$ is a metric space and that $p\ge 1$.
By a slight abuse of notation, we write $d$ for both the metric on $\mathcal{X}$ and also for the topology on $\mathcal{X}$ generated by this metric.
We also note that we have
\begin{equation}\label{eqn:dist-comparison}
d^p(x',y) \le 2^{p-1}(d^p(x',x'')+d^p(x'',y))
\end{equation}
for all $x',x'',y\in \mathcal{X}$, which can be easily seen from the convexity of the map $t\mapsto t^{p}$.

\subsection{Spaces of subsets}\label{subsec:subsets}

For a Hausdorff topology $\tau$ on $\mathcal{X}$, we write $\closed^{\tau}(\mathcal{X})$ for the collection of all $\tau$-closed subsets of $\mathcal{X}$, and we also write $C(\mathcal{X}):=C^d(\mathcal{X})$ for the collection of all $d$-closed subsets of $X$.
Additionally, we write $\cbdd(\mathcal{X}):=\cbdd^d(\mathcal{X})$ for the collection of all non-empty $d$-compact subsets of $\mathcal{X}$.
In this subsection we introduce some notions of convergence for such spaces of subsets and establish some basic properties that we will later need.
We direct the interested reader to \cite[Chapter~5]{Beer} for further information on set-valued analysis.

To begin we introduce and review some basic properties of ``Kuratowski convergence'', sometimes called the ``Kuratowski-Painlev\'e convergence''.
For $\{C_n\}_{n\in\N}$ arbitrary subsets of $X$ and for $\tau$ any Hausdorff topology on $\mathcal{X}$, we define the sets
\begin{align*}
\underset{n\to\infty}{\kurouter^{\tau}}C_n &:= \left\{x\in \mathcal{X}:{\begin{matrix}{\mbox{for all $\tau$-open neighborhoods }}U{\mbox{ of }}x,\\ U\cap C_{n}\neq \emptyset \mbox{ for  infinitely many } n\in\N\end{matrix}}\right\} \\
\underset{n\to\infty}{\kurinner^{\tau}}C_n &:= \left\{x\in \mathcal{X}:{\begin{matrix}{\mbox{for all $\tau$-open neighborhoods }}U{\mbox{ of }}x,\\U\cap C_{n}\neq \emptyset {\mbox{ for large enough }}n \in \N\end{matrix}}\right\}.
\end{align*}
called the \textit{$\tau$-Kuratowski upper limit} and \textit{$\tau$-Kuratowski lower limit}, respectively.
More concretely, for a point $x\in \mathcal{X}$, we have $x\in \kurouter^{\tau}_{n\in\N}C_n$ if and only if there exists some subsequence $\{n_j\}_{j\in\N}$ and some $x_j\in C_{n_j}$ for each $j\in\N$ such that $x_{j}\to x$ in $\tau$, and we have $x\in \kurinner^{\tau}_{n\in\N}C_n$ if and only if for any subsequence $\{n_j\}_{j\in\N}$ there exists a further subsequence $\{j_i\}_{i\in\N}$ and some $x_i\in C_{n_{j_i}}$ for each $i\in\N$ with $x_i\to x$ in $\tau$.

If $\{C_n\}_{n\in\N}$ and $C$ are such that $\kurouter^{\tau}_{n\in\N}C_n\subseteq C$, then we say that $\{C_n\}_{n\in\N}$ \textit{converges in the $\tau$-Kuratowski upper sense to} $C$, and if they are such that $\kurinner^{\tau}_{n\in\N}C_n\supseteq C$, then we say that $\{C_n\}_{n\in\N}$ \textit{converges in the $\tau$-Kuratowski lower sense to} $C$.
Note the inclusions rather than exact equality here.
If $\{C_n\}_{n\in\N}$ satisfies $\kurouter^{\tau}_{n\in\N}C_n=\kurinner^{\tau}_{n\in\N}C_n = C$, then we write $\kurlimit^{\tau}_{n\in\N}C_n$ for their common value, called the \textit{$\tau$-Kuratowski (full) limit}, and in this case we say that $\{C_n\}_{n\in\N}$ \textit{converges in the $\tau$-Kuratowski (full) sense to} $C$ or that \textit{the $\tau$-Kuratowski limit exists} and equals $C$.

It is known that convergence in the $\tau$-Kuratowski full sense coincides with convergence in a topology on $\closed^{\tau}(\mathcal{X})$ if and only if $(\mathcal{X},\tau)$ is locally compact, in which case the topology is exactly the Fell topology (see \cite[Theorem~5.2.6]{Beer} and the remarks thereafter).
Contrarily, we do not know when convergence in the Kuratowski upper and lower senses correspond to convergence in suitable topologies.
Our later results will establish that certain functions map Kuratowski-convergent sequences to Kuratowski-convergent sequences, but, because of this issue of possible non-topologizability, we will not attempt to say that such functions are continuous;
we will, however, occasionally say that such functions are ``continuous'' in the exposition, where the quotation marks are meant to suggest that the term is being used non-rigorously for intuition only.
It is possible, in the locally compact setting, that these  ``continuity'' statements become bona fide continuity statements, but we will not concern ourselves with such technicalities in this work.

Next, we introduce and review some basic properties of ``Hausdorff  convergence'', and we note that a more comprehensive account of this theory can be found in \cite[Section~7.3]{BBI}.
That is, for $x\in \mathcal{X}$ and $C'\in \cbdd(\mathcal{X})$, write
\begin{equation*}
d(x,C') := \min_{x'\in C'}d(x,x')
\end{equation*}
for the shortest distance from the point $x$ to the set $C'$.
Observe that $C'$ being non-empty and $d$-compact implies that $d(x,C') < \infty$ and that the infimum is achieved.
Now for $C,C'\in\cbdd(X)$, write
\begin{equation*}
\dvecHaus(C,C') := \max_{x\in C}d(x,C') = \max_{x\in C}\min_{x'\in C'}d(x,x')
\end{equation*}
for the largest possible shortest distance from a point in $C$ to the set $C'$.
Observe that $C$ and $C'$ being non-empty and $d$-compact imply that $\dvecHaus(C,C') < \infty$ and that the supremum and infimum are both achieved.
Although $\dvecHaus$ is not a metric (it is not symmetric), it it easy to show that it satisfies the following type of triangle inequality
\begin{equation}\label{dvecHaus-triangle-ineq}
\dvecHaus(C,C'') \le \dvecHaus(C,C') + \dvecHaus(C',C'')
\end{equation}
for $C,C',C''\in \cbdd(\mathcal{X})$.

If $\{C_n\}_{n\in\N}$ and $C$ in $\cbdd(X)$ are such that $\dvecHaus(C_n,C)\to 0$, then we say that $\{C_n\}_{n\in\N}$ \textit{converges in the $d$-Hausdorff upper sense to} $C$, and if they are such that $\dvecHaus(C,C_n)\to 0$, then we say that $\{C_n\}_{n\in\N}$ \textit{converges in the $d$-Hausdorff lower sense to} $C$.
As before, we will not discuss these as topological notions (although we observe, in contrast to the Kuratowski convergences, that the Hausdorff convergences are always topologizable \cite[Lemma~3.12]{JaffeInfDim}).

\begin{remark}\label{rem:dvecHaus-finite}
	The definition of $\dvecHaus$ immediately extends from to the case of $C,C'\in \cbdd(\mathcal{X})$ to the case of non-empty $C,C'\in \closed(\mathcal{X})$, provided that we replace the maximum and minimum with supremum and infimum and that we allow it to take values in the extended real half-line, $[0,\infty]$.
	So, if $\{C_n\}_{n\in\N}$ and $C$ in $\closed(\mathcal{X})$ are assumed only to be non-empty then the expression $\dvecHaus(C_n,C) \to 0$ is taken to mean that $\dvecHaus(C_n,C)$ is finite for sufficiently large $n\in\N$ and that it converges to zero as $n\to\infty$.
\end{remark}

Now, for $C,C'\in \cbdd(\mathcal{X})$, write
\begin{equation*}
\dHaus(C,C') := \max\{\dvecHaus(C,C'),\dvecHaus(C',C)\}.
\end{equation*}
This is a well-studied object called the \textit{Hausdorff metric}, and it is, as the name suggests, a bona fide metric on $\cbdd(\mathcal{X})$.
Write $\hausdorfftopo$ for the topology on $\cbdd(\mathcal{X})$ such that $\{C_n\}_{n\in\N}$ and $C$ in $\cbdd(\mathcal{X})$ have $\lim_{n\to\infty}C_n = C$ in $\hausdorfftopo$ if and only if $\lim_{n\to\infty}\dHaus(C_n,C)  = 0$; this is called the \textit{Hausdorff topology}.
Equivalently, $\hausdorfftopo$ is the weakest topology on $\cbdd(\mathcal{X})$ such that $\dvecHaus(C,\cdot),\dvecHaus(\cdot,C):\cbdd(\mathcal{X},d)\to [0,\infty)$ are continuous for all $C\in \cbdd(\mathcal{X})$.

Next we give some simple but important results related to the aforementioned concepts.

\begin{lemma}\label{lem:Kur-lower-subseq}
	If $\{R_n\}_{n\in\N}$ in $\closed(\mathcal{X})$ and $C$ in $\cbdd(\mathcal{X})$ have $C\subseteq \kurinner^{d}_{n\in\N}R_n$ and $\#C<\infty$, then, there exists a subsequence $\{n_j\}_{j\in\N}$ and $C_j\subseteq R_{n_j}$ with $\#C_{j} = \#C$ for all $j\in\N$ such that we have $\dHaus(C_{j}, C)\to 0$.
\end{lemma}

\begin{proof}
Write $C = \{x_1,\ldots, x_k\}$ for $k := \#C$, and set $\varepsilon:= \min\{d(x_{\ell},x_{\ell'}): 1\le \ell<\ell'\le k\}> 0$.
Then apply $C\subseteq \kurinner^{d}_{n\in\N}R_n$ iteratively $k$ times to get a subsequence $\{n_j\}_{j\in\N}$ and sets $C_j := \{x_{1,j},\ldots, x_{k,j}\}\subseteq R_{n_j}$ for each $j\in\N$ such that we have $x_{\ell,j}\to x_{\ell}$ in $d$ as $j\to\infty$ for all $1\le \ell \le k$.

Next, observe that we have $\max_{1\le \ell\le k}d(x_{\ell,j},x_\ell) < \varepsilon/3$ for sufficiently large $j\in\N$.
Also by the triangle inequality, we have $\varepsilon \le d(x_{\ell},x_{\ell'})\le d(x_{\ell},x_{\ell,j}) + d(x_{\ell,j},x_{\ell',j})+d(x_{\ell',j},x_{\ell'})$ for all $1\le \ell<\ell'\le k$.
It follows that, for sufficiently large $j\in\N$ and all $1\le \ell<\ell'\le k$, we have $0 < \frac{\varepsilon}{3} \le d(x_{\ell,j},x_{\ell',j})$.
This implies that $\#C_j = k$ for sufficiently large $j\in\N$.

By construction we have $\dvecHaus(C,C_j)\to 0$, so it only remains to show $C_j\to C$ in $\dHaus$ as $j\to\infty$.
To do this, note that for sufficiently large $j\in\N$ we have $\dvecHaus(C,C_j) < \varepsilon/2$ hence we can construct a map $\phi_j:C\to C_j$ by sending each $x\in C$ to some $\phi_j(x)\in C_j$ such that $d(x,\phi_j(x)) < \varepsilon/2$.
Of course, if $\phi_j(x) = \phi_j(x')$ for $x,x'\in C$ then we have $d(x,x')\le d(x,\phi_j(x)) + d(x',\phi_j(x')) < \varepsilon$, hence $x=x'$ by the minimality of $\varepsilon$, hence $\phi_j:C\to C_j$ is injective.
Since an injective map between finite sets of the same cardinality is automatically a bijection, there exists, for sufficiently large $j\in\N$, a well-defined function $\phi_j^{-1}:C_j\to C$ such that for all $x_j\in C_j$ we have $d(x_j,\phi_j^{-1}(x_j)) < \varepsilon/2$.
This means we have $\dvecHaus(C_j,C) < \varepsilon/2$ for sufficiently large $j\ge N$, as needed.
\end{proof}

\begin{lemma}\label{lem:joint-cty}
	If $\{y_n\}_{n\in\N}$ and $y$ in $\mathcal{X}$ have $y_n\to y$ and $\{S_n\}_{n\in\N}$ and $S$ in $\cbdd(\mathcal{X})$ have $S_n\to S$ in $\dHaus$, then $d(y_n,S_n)\to d(y,S)$.
\end{lemma}

\begin{proof}
	By \eqref{dvecHaus-triangle-ineq}, we have
	\begin{equation*}
		d(y_n,S_n) \le d(y_n,S) + \dvecHaus(S,S_n)
	\end{equation*}
	and
	\begin{equation*}
		d(y_n,S) \le d(y_n,S_n) + \dvecHaus(S_n,S).
	\end{equation*}
	Thus, the result follows from the well-known continuity of the map $d(\,\cdot\,,S):\mathcal{X}\to\R$.
\end{proof}

Another definition we will need is that of a \textit{random set}, by which we mean a measurable map $S:(\Omega,\F)\to (\closed(\mathcal{X}),\mathcal{E}(\mathcal{X}))$ where $\closed(\mathcal{X})$ is the collection of all $d$-closed subsets of $X$ and where $\mathcal{E}(\mathcal{X})$ is \textit{Effros $\sigma$-algebra}, that is, the $\sigma$-algebra on $\closed(\mathcal{X})$ generated by $\{\{C\in \closed(\mathcal{X}): C\cap B_{r}(x) = \emptyset\}: x\in \mathcal{X}, r>0\}$; here $B_{r}(x):= \{y\in \mathcal{X}: d(x,y)<r\}$ denotes the open ball of radius $r$ around $x$.
By a \textit{random $\tau$-closed set} we mean a random set $S$ satisfying $\P(S \textnormal{ is $\tau$-closed}) = 1$.

Lastly, we introduce a bit of notation.
Suppose that $\{S_n\}_{n\in\N}$ and $S$ are finite subsets of $\mathcal{X}$ with $\#S_n = \#S =: k$ for sufficiently large $n\in\N$.
Then, we write $S_n\to S$ in $\tau$ to mean that there exist labelings $S_n = \{a^{n}_{1},\ldots, a^{n}_{k}\}$ for sufficiently large $n\in\N$ and $S = \{a_1,\ldots, a_k\}$ such that we have $a_{\ell}^{n}\to a_{\ell}$ in $\tau$ for all $1\le \ell \le k$.
Similarly, we write $S_n\to S$ in $d$ to mean that there exist labelings as above such that we have $a_{\ell}^{n}\to a_{\ell}$ in $d$ for all $1\le \ell \le k$.
Note also that for finite sets $\{S_n\}_{n\in\N}$ and $S$ of the same cardinality, we have that $\dHaus(S_n,S)\to 0$ is equivalent to $S_n\to S$ in $d$.

\subsection{Spaces of measures}

Write $\mathcal{P}(\mathcal{X})$ to denote the space of Borel probability measures on $\mathcal{X}$, with respect to the topology generated by $d$.
We write $\weak$ for the topology on $\mathcal{P}(\mathcal{X})$ such that $\{\mu_n\}_{n\in\N}$ and $\mu$ in $\mathcal{P}(\mathcal{X})$ have $\mu_n\to \mu$ in $\weak$ if and only if we have $\int_{\mathcal{X}}f\, \diff\mu_n\to \int_{\mathcal{X}}f \, \diff\mu$ for all bounded, $d$-continuous functions $f:\mathcal{X}\to \R$; this is nothing more than the weak topology corresponding to the topology generated by $d$.
Then write $\mathcal{P}_p(\mathcal{X})\subseteq \mathcal{P}(\mathcal{X})$ for the subspace of all Borel probability measures which satisfy $\int_{\mathcal{X}}d^p(x,y)\, \diff\mu(y) < \infty$ for some $x\in X$.
By \eqref{eqn:dist-comparison}, we have $\mu\in\mathcal{P}_p(\mathcal{X})$ if and only if $\int_{\mathcal{X}}d^p(x,y)\, \diff\mu(y) < \infty$ for all $x\in \mathcal{X}$.

Next define the function $W_p:\cbdd(\mathcal{X})\times \mathcal{P}_p(\mathcal{X})\to [0,\infty)$ via
\begin{equation*}
W_p(S,\mu) := \int_{\mathcal{X}}d^p(y,S)\, \diff\mu(y) = \int_{\mathcal{X}}\min_{x\in S}d^p(x,y)\, \diff\mu(y),
\end{equation*}
for $S\in\cbdd(\mathcal{X})$ and $\mu\in\mathcal{P}_p(\mathcal{X})$.

Now define a topology $\tau_{w}^{p}$ on $\mathcal{P}_p(\mathcal{X})$ such that $\{\mu_n\}_{n\in\N}$ and $\mu$ in $\mathcal{P}_p(\mathcal{X})$ have $\mu_n \to \mu$ in $\tau_{w}^{p}$ if and only if we have $\mu_n\to\mu$ in $\weak$ and $\int_{\mathcal{X}}d^p(x,y)\, \diff\mu_n(y) \to \int_{\mathcal{X}}d^p(x,y)\, \diff\mu(y)$ as $n\to\infty$ for some $x\in \mathcal{X}$; in fact, it is known \cite[Lemma~2.1]{EvansJaffeFrechet} that this is equivalent to $\mu_n\to\mu$ in $\weak$ and $\int_{\mathcal{X}}d^p(x,y)\, \diff\mu_n(y) \to \int_{\mathcal{X}}d^p(x,y)\, \diff\mu(y)$ as $n\to\infty$ for all $x\in \mathcal{X}$,
We refer to $\tau_{w}^{p}$ as the \textit{$p$-Wasserstein topology} and we note that it is is closely related to many aspects of optimal transport \cite{Villani}.

Next we develop a few fundamental approximation results, whose proofs follow a common theme:
First, we use Skorokhod's representation theorem \cite[Theorem~4.30]{Kallenberg} to represent convergence in $\weak$ as almost sure convergence on a suitable probability space.
Second, we use the fact \cite[Lemma~5.11]{Kallenberg} that random variables converging almost surely have converging expectations if and only if they are uniformly integrable.
Lastly, we use the inequality \eqref{eqn:dist-comparison} to provide some dominations which allow us to transfer uniform integrability across different random variables.

\begin{lemma}\label{lem:Skor}
	Suppose that $\{\mu_n\}_{n\in\N}$ and $\mu$ in $\mathcal{P}_p(\mathcal{X})$ have $\mu_n\to \mu$ in $\weakp$.
	Then, there exists a probability space $(\Omega,\F,\P)$, with expectation $\E$, on which are defined random variables $\{Y^n\}_{n\in\N}$ and $Y$ with laws $\{\mu_{n}\}_{n\in\N}$ and $\mu$, respectively, such that we have $Y^n\to Y$ almost surely and $\E[d^p(Y^n,Y)]\to 0$ as $n\to\infty$.
\end{lemma}

\begin{proof}
	By the standard Skorokhod theorem, we get a sequence of random variables with desired laws and with the desired almost sure convergence property, so it only remains to show that this coupling satisfies $\E[d^p(Y^n,Y)]\to 0$ as $n\to\infty$.
	Indeed, we have $d^p(Y^n,Y)\to 0$ almost surely as $n\to\infty$, so it suffices to show that the family $\{d^p(Y^n,Y)\}_{n\in\N}$ is uniformly integrable.
	To do this, fix arbitrary $x\in X$ and use \eqref{eqn:dist-comparison} to bound:
	\begin{equation*}
		d^p(Y^n,Y)  \le 2^{p-1}\left(d^p(x,Y^n)+d^p(x,Y)\right).
	\end{equation*}
	In particular, it further suffices to show that $\{d^p(x,Y^n)\}_{n\in\N}$ is uniformly integrable.
	Since $d^p(x,Y^n)\to d^p(x,Y)$ almost surely, this is equivalent to $\E[d^p(x,Y^n)]\to \E[d^p(x,Y)]$ which follows from the definition of convergence in $\weakp$.
	This finishes the proof.
\end{proof}

\begin{lemma}\label{lem:W-cts}
	The function $W_p:(\cbdd(\mathcal{X})\times \mathcal{P}_p(\mathcal{X}),\dHaus\times\tau_{w}^{p})\to [0,\infty)$ is continuous.
\end{lemma}

\begin{proof}
	Suppose $\{(S_n,\mu_n)\}_{n\in\N}$ and $(S,\mu)$ in $\cbdd(\mathcal{X})\times \mathcal{P}_p(\mathcal{X})$ have $(S_n,\mu_n)\to (S,\mu)$ in $\dHaus\times \weakp$, and let $(\Omega,\F,\P)$ be as in Lemma~\ref{lem:Skor}, and write $\E$ for the corresponding expectation operator.
	Then Lemma~\ref{lem:joint-cty} implies $d(Y^n,S_n)\to d(Y,S)$ almost surely.
	Now fix $x\in S$ and note that $\{d(x,Y^n)\}_{n\in\N}$ is uniformly integrable by the definition of $\weakp$.
	By \eqref{eqn:dist-comparison} we have
	\begin{equation*}
		d^p(Y^n,S_n) \le 2^{p-1}\left(d^p(x,Y^n) + \left(\dvecHaus(S,S_n)\right)^p\right),
	\end{equation*}
	so it follows that $\{d^p(Y^n,S_n)\}_{n\in\N}$ is uniformly integrable.
	Consequently, we have
	\begin{equation*}
		W_p(S_n,\mu_n) = \E\left[d^p(Y^n,S_n)\right] \to  \E\left[d^p(Y,S)\right] = W_p(S,\mu)
	\end{equation*}
	as claimed.
\end{proof}

\begin{lemma}\label{lem:partial-cty}
	If $\{S_n\}_{n\in\N}$ in $\cbdd(\mathcal{X})$ and $\{\mu_n\}_{n\in\N}$ and $\mu$ in $\mathcal{P}_p(\mathcal{X})$ have $\mu_n\to \mu$ in $\weakp$, then
	\begin{equation*}
		\liminf_{n\to\infty}W_p(S_n,\mu_n) = \liminf_{n\to\infty}W_p(S_n,\mu)
	\end{equation*}
	and
	\begin{equation*}
		\limsup_{n\to\infty}W_p(S_n,\mu_n) = \limsup_{n\to\infty}W_p(S_n,\mu).
	\end{equation*}
\end{lemma}

\begin{proof}
	Recall \cite[Lemma~2.3]{EvansJaffeFrechet} that for all $\varepsilon>0$ there exists $c_{p,\varepsilon}>0$ such that we have
	\begin{equation}\label{eqn:Peter-Paul}
		d^p(x,x'') \le c_{p,\varepsilon}d^p(x,x') + (1+\varepsilon)d^p(x',x'')
	\end{equation}
	for all $x,x',x''\in \mathcal{X}$ hence	
	\begin{equation}\label{eqn:Peter-Paul-sets}
		\left(\dvecHaus(C,C'')\right)^p \le c_{p,\varepsilon}\left(\dvecHaus(x,x')\right)^p + (1+\varepsilon)\left(\dvecHaus(x',x'')\right)^p
	\end{equation}
	for all $C,C',C''\in\cbdd(\mathcal{X})$.
	Also get a probability space $(\Omega,\F,\P)$ with expectation $\E$ and random variables $\{Y^n\}_{n\in\N}$ and $Y$ as in Lemma~\ref{lem:Skor}.
	Then, for arbitrary $\varepsilon>0$, use \eqref{eqn:Peter-Paul-sets} to get:
	\begin{align*}
		W_p(S_n,\mu_n) &= \E\left[d^p(Y^n,S_n)\right] \\
		&\le c_{p,\varepsilon}\E\left[d^p(Y^n,Y)\right] + (1+\varepsilon)\E\left[d^p(Y,S_n)\right] \\
		&= c_{p,\varepsilon}\E\left[d^p(Y^n,Y)\right] + (1+\varepsilon)W_p(S_n,\mu).
	\end{align*}
	Then take $n\to\infty$ and $\varepsilon\to 0$ to get
	\begin{equation}\label{eqn:Skor-1}
		\liminf_{n\to\infty}W_p(S_n,\mu_n) \le \liminf_{n\to\infty}W_p(S_n,\mu).
	\end{equation}
	Similarly, we can bound, for any $\varepsilon>0$:
	\begin{align*}
		W_p(S_n,\mu) &= \E\left[d^p(Y,S_n)\right] \\
		&\le c_{p,\varepsilon}\E\left[d^p(Y,Y^n)\right] + (1+\varepsilon)\E\left[d^p(Y^n,S_n)\right] \\
		&= c_{p,\varepsilon}\E\left[d^p(Y,Y^n)\right] + (1+\varepsilon)W_p(S_n,\mu_n).
	\end{align*}
	Thus,
	\begin{equation}\label{eqn:Skor-2}
		\liminf_{n\to\infty}W_p(S_n,\mu) \le \liminf_{n\to\infty}W_p(S_n,\mu_n).
	\end{equation}
	Combining \eqref{eqn:Skor-1} and \eqref{eqn:Skor-2} gives the desired $\liminf$ equality, and the $\limsup$ equality is proved in the same way.

\end{proof}

\section{Analysis of the clustering map}\label{sec:cluster}

In this section, we show that various natural maps used in the construction of clustering procedures are ``continuous'' with respect to various topologies of spaces of measures, space of sets, and on spaces of sets of sets.
In particular, we show Proposition~\ref{prop:cpt-operator} and Proposition~\ref{prop:clustering-cty} which provide ``continuity'' of the clustering map, and we state Proposition~\ref{prop:elbow-k-cts} which provides sufficient conditions for the continuity of the adaptive choice of $k$ arising in the elbow-method.
Throughout this section, we fix $(\mathcal{X},d)$ a metric space and $p\ge 1$ an arbitrary exponent.

To begin, we give the basic notions of the clustering procedures of interest.

\begin{definition}
	For $\mu\in\mathcal{P}_p(\mathcal{X})$, $k\in\N$, and $R\in\closed(\mathcal{X})$, set
	\begin{equation}\label{eqn:m-def}
	m_{k,p}(\mu,R) := \inf_{\substack{S'\subseteq R \\ 1\le\#S'\le k}}W_p(S',\mu),
	\end{equation}
	and also set $C_{p}(\mu,k,R)$ to be the set of all $S\subseteq R$ with $1\le\#S\le k$ satisfying
	\begin{equation}\label{eqn:cluster-def}
	W_p(S,\mu)\le m_{k,p}(\mu,R).
	\end{equation}
	If a set $S\subseteq \mathcal{X}$ has $S\subseteq R$ and $1\le\#S\le k$ it is called \textit{feasible} and if it achieves \eqref{eqn:cluster-def} it is called \textit{optimal}.
	Note that $C_p(\mu,k,R)$ is empty if there are no optimal sets or if $R$ is empty.
	We refer to $C_{p}(\mu,k,R)$ as the \textit{set of sets of $R$-restricted $(k,p)$-means clustering centers}.
	For $\varepsilon\ge 0$, we also set $C_{p}(\mu,k,R;\varepsilon)$ to be the set of all $S\subseteq R$ with $1\le\#S\le k$ satisfying
	\begin{equation}
		W_p(S,\mu)\le m_{k,p}(\mu,R) + \varepsilon,
	\end{equation}
	called the \textit{set of sets of $\varepsilon$-relaxed $R$-restricted $(k,p)$-means clustering centers}.
\end{definition}

We also recall the assumptions (W1), (W2), and (W3) stated in the introduction of the paper.
The importance of assumptions (W1) and (W2) is that they lead to the following technical result, which forms the basis for much of our later work.

\begin{proposition}\label{prop:structural}
	Suppose that $\tau$ satisfies \textnormal{(W1)} and \textnormal{(W2)}, that $\{R_n\}_{n\in\N}$ and $R$ in $\closed^{\tau}(\mathcal{X})$ satisfy $\kurouter_{n\in\N}^{\tau}R_n\subseteq R$, and that $\{S_n\}_{n\in\N}$ satisfy $S_n\subseteq R_n$ and $\#S_n= k$ for all $n\in\N$.
	Then, for any labeling $S_n = \{a_1^{n},\ldots, a_{k}^{n}\}$ for all $n\in\N$, there exists a subsequence $\{n_j\}_{j\in\N}$ such that, for each $1\le \ell \le k$, exactly one of
	\begin{itemize}
		\item $\{a_{\ell}^{n_j}\}_{j\in\N}$ is $d$-unbounded, or
		\item $\{a_{\ell}^{n_j}\}_{j\in\N}$ is $\tau$-convergent to some $a_{\ell}\in R$
	\end{itemize}
	holds.
	Consequently, the set $S:=\{a_{\ell}: 1\le \ell \le k, \{a_{\ell}^{n_j}\}_{j\in\N} \textnormal{ is $d$-bounded}\}$ satisfies $S\subseteq R$ and $\#S\le k$, and we have
	\begin{equation*}
		d^p(y,S)\le \liminf_{j\to\infty}d^p(y,S_{n_j})
	\end{equation*}
	for all $y\in \mathcal{X}$, hence
	\begin{equation*}
		W_p(S,\mu)\le \liminf_{j\to\infty}W_p(S_{n_j},\mu)
	\end{equation*}
	for all $\mu\in\mathcal{P}_p(\mathcal{X})$.
	If additionally $\#S=k$, then $S_{n_j}\to S$ in $\tau$.
\end{proposition}

\begin{proof}
	We construct $\{n_j\}_{j\in\N}$ by iteratively applying (W1):
	First, if $A_1:=\{a_1^n\}_{n\in\N}$ is $d$-bounded, we use (W1) to get $\{n_{1,j}\}_{j\in\N}$ and $a_1\in \mathcal{X}$ with $a_1^{n_{1,j}}\to a_1$ in $\tau$.
	Then recursively for $1< \ell \le k$, if $A_{\ell}:=\{a_{\ell}^{n_{\ell-1,j}}\}_{n\in\N}$ is $d$-bounded, we use (W1) to get $\{n_{\ell,j}\}_{j\in\N}$ and $a_{\ell}\in \mathcal{X}$ with $a_{\ell}^{n_{\ell,j}}\to a_{\ell}$ in $\tau$.
	Now we take $\{n_j\}_{j\in\N}:= \{n_{k,j}\}_{j\in\N}$, and we set
	\begin{equation*}
		S_j':=\{a_{\ell}^{n_j}: 1 \le \ell \le k, A_{\ell} \textnormal{ is $d$-bounded}\}
	\end{equation*}
	for $j\in\N$.
	Observe that we immediately have $S\subseteq R$ and $\#S\le k$.
	Moreover, if $\#S=k$ then it follows by definition that we have $S_{n_j}\to S$ in $\tau$.
	In any case, it follows that we have $S_j'\subseteq S_{n_j}$ for all $j\in\N$, as well as
	\begin{equation*}
		\liminf_{j\to\infty}d^p(y,S_{j}') = \liminf_{j\to\infty}d^p(y,S_{n_j}).
	\end{equation*}
	for all $y\in X$, since all elements of $S_{n_j}\setminus S_j'$ come from $d$-unbounded sequences as $j\to\infty$.

	To show the inequalities, we take arbitrary $y\in X$.
	Let $\{j_i\}_{i\in\N}$ be a subsequence satisfying
	\begin{equation*}
		\liminf_{j\to\infty}d^p(y,S_{j}') = \lim_{i\to\infty}d^p(y,S_{j_i}').
	\end{equation*}
	By the pigeonhole principle, there exists some $1\le \ell\le k$ and a further subsequence $\{i_u\}_{u\in\N}$ such that
	\begin{equation*}
		d^p(y,S_{j_{i_u}}') = d^p(a_{\ell}^{n_{j_{i_u}}},y)
	\end{equation*}
	for all $u\in\N$.
	Therefore, by (W2) and the construction, we get:
	\begin{align*}
		d^p(y, S) \le d^p(a_{\ell},y) \le \liminf_{u\to\infty}d^p(a_{\ell}^{n_{j_{i_u}}},y) &= \lim_{u\to\infty}d^p(y,S_{j_{i_u}}') \\
		&= \liminf_{j\to\infty}d^p(y,S_{j}') \\
		&= \liminf_{j\to\infty}d^p(y,S_{n_j}).
	\end{align*}
	Finally, we apply Fatou to get
	\begin{align*}
		W_p(S,\mu) &\le \int_{\mathcal{X}}\liminf_{j\to\infty}d^p(y,S_{n_j})\, \diff\mu(y) \\
		&\le \liminf_{j\to\infty}\int_{\mathcal{X}}d^p(y,S_{n_j})\, \diff\mu(y) \le \liminf_{j\to\infty}W_p(S_{n_j},\mu),
	\end{align*}
	as claimed.
\end{proof}

\begin{lemma}\label{lem:non-empty}
	If $\tau$ satisfies \textnormal{(W1)} and \textnormal{(W2)} and $(\mu,k,R)\in\mathcal{P}_p(\mathcal{X})\times\N\times \closed^{\tau}(\mathcal{X})$, then $C_p(k,\mu,R)$ is non-empty.
\end{lemma}

\begin{proof}
For each $n\in\N$, let $S_n$ be a set of $2^{-n}$-relaxed $R$-restricted $(k,p)$-means cluster centers.
That is, we have $S_n\subseteq R$ and $1\le \#S_n\le k$, as well as $W_p(S_n,\mu) \le m_{k,p}(\mu,R)+2^{-n}$ for all $n\in\N$.
We can also assume that $\#S_n = k$ by adding more points if necessary, since this cannot increase the objective.
Now get $\{n_j\}_{j\in\N}$ and $S\subseteq R$ as in Proposition~\ref{prop:structural}, and note that this implies
\begin{equation*}
	W_p(S,\mu) \le \liminf_{j\to\infty}W_p(S_{n_j},\mu) \le \liminf_{j\to\infty}(m_{k,p}(\mu,R)+2^{-n_j}) = m_{k,p}(\mu,R).
\end{equation*}
This shows $S\in C_p(\mu, k, R)$, so $C_p(\mu, k, R)$ is non-empty.
\end{proof}

The remainder of our results require an important notion of non-singularity, which has been introduced in \cite{Lember,Parna1,Parna2}.
To understand it, consider any $(\mu,k,R)\in \mathcal{P}_p(\mathcal{X})\times\N\times \closed(\mathcal{X})$.
Notice that the infimum in \eqref{eqn:m-def} can equivalently be taken over all $S'\subseteq R$ with $\#S'=k$, since adding points to a set of cluster centers can never increase its objective.
However, it is possible that a set of cluster centers $S'\subseteq R$ with $\#S'<k$ is already optimal.
The following notion excludes this possibility:

\begin{definition}
	We say that $(\mu,k,R)\in\mathcal{P}_p(\mathcal{X})\times \N\times \closed(\mathcal{X})$ is \textit{non-singular} if $m_{1,p}(\mu,R)>m_{2,p}(\mu,R)>\cdots >m_{k,p}(\mu,R)$ and \textit{singular} otherwise.
\end{definition}

Observe for $R=\emptyset$ that we have $m_{k,p}(\mu,\emptyset) = \infty$ for all $k\in\N$, so $(\mu,k,\emptyset)$ can never be non-singular.
In other words, $(\mu,k,R)$ being non-singular implies that $R$ is non-empty.
We also give the following simple sufficient condition for non-singularity:

\begin{lemma}\label{lem:non-singular}
	If $\tau$ satisfies \textnormal{(W1)} and \textnormal{(W2)} and $(\mu,k,R)\in\mathcal{P}_p(\mathcal{X})\times\N\times \closed^{\tau}(\mathcal{X})$ satisfies $k\le \#(R\cap\supp(\mu))$, then $(\mu,k,R)$ is non-singular.
\end{lemma}

\begin{proof}
Fix $1 < \ell\le k$, and use Lemma~\ref{lem:non-empty} to get $S\in C_p(\mu, \ell-1, R)$.
Then $\#S\le \ell-1 <k$ and $\#(R\cap \supp(\mu))\ge k$ together imply that there is some $z\in (R\cap \supp(\mu))\setminus S$.
In particular, the set $S\cup\{z\}$ satisfies $S\cup\{z\}\subseteq R$ and $1\le \#(S\cup\{z\})\le \ell$.
Thus, the proof is complete if we can show that we have $W_p(S\cup\{z\},\mu) < W_p(S,\mu)=m_{\ell-1,p}(\mu,R)$.
To do this, assume for the sake of contradiction that we have $W_p(S\cup\{z\},\mu) = W_p(S,\mu)$.
Since we of course have $d^p(y,S\cup\{z\})\le d^p(y,S)$ for all $y\in X$, we conclude that we must have $d^p(y,S\cup\{z\})= d^p(y,S)$ for all $\mu$-almost all $y\in \mathcal{X}$.
Thus, choosing $r>0$ small enough so that we have $d^p(y,S\cup\{z\}) = d^p(z,y)$ for all $y\in B_r(z)$, we find $\mu(B_r(z)) = 0$, and this contradicts $z\in \supp(\mu)$.
\end{proof}
	
One important consequence of non-singularity is that it guarantees the following uniform boundedness property; this will help us verify the hypotheses of Proposition~\ref{prop:structural} in our later work.
Its proof is just a slight strengthening of \cite{Pollard, Lember}, so we defer the details to the supplementary material.

\begin{proposition}\label{prop:bdd}
	Suppose that $\{(\mu_n,k_n,R_n)\}_{n\in\N}$ and $(\mu,k,R)$ in $\mathcal{P}_p(\mathcal{X})\times \N\times \closed(\mathcal{X})$ are such that $(\mu,k,R)$ is non-singular, and
	\begin{itemize}
		\item $\mu_n\to \mu$ in $\weakp$,
		\item $k_n\neq k$ for finitely many $n\in\N$, and
		\item $R\subseteq \kurinner^d_{n\in\N}R_n$.
	\end{itemize}
	Also suppose that non-negative constants $\{\varepsilon_n\}_{n\in\N}$ have $\varepsilon_n\to 0$.
	Then, there exists some $z\in \mathcal{X}$ and $r > 0$ such that for sufficiently large $n\in\N$ and all $S_n\in C_p(\mu_n,k_n,R_n;\varepsilon_n)$, we have $\#S_n = k$ and $S_n\subseteq B_{r}(z)$.
\end{proposition}

Now we can prove the main result of this subsection.

\begin{proposition}\label{prop:cpt-operator}
	Let $(\mathcal{X},d)$ be separable, and let $\{(\mu_n,k_n,R_n)\}_{n\in\N}$ and $(\mu,k,R)$ in $\mathcal{P}_p(\mathcal{X})\times \N\times \closed^{\tau}(\mathcal{X})$ be such that $(\mu,k,R)$ is non-singular, and
	\begin{itemize}
		\item $\mu_n\to \mu$ in $\weakp$,
		\item $k_n\neq k$ for finitely many $n\in\N$, and
		\item $\kurouter^{\tau}_{n\in\N}R_n \subseteq R \subseteq \kurinner^{d}_{n\in\N}R_n$.
	\end{itemize}
	If $\tau$ satisfies \textnormal{(W1)} and \textnormal{(W2)}, then, if $S_n\in C_p(\mu_n,k_n,R_n)$ for all $n\in\N$, there exists $\{n_j\}_{j\in\N}$ and $S\in C_p(\mu,k,R)$ such that $S_{n_j}\to S$ in $\tau$ as $j\to\infty$.
	If $\tau$ satisfies \textnormal{(W1)}, \textnormal{(W2)}, and \textnormal{(W3)}, then, if $S_n\in C_p(\mu_n,k_n,R_n)$ for all $n\in\N$, there exists $\{n_j\}_{j\in\N}$ and $S\in C_p(\mu,k,R)$ such that hence $S_n\to S$ in $d$ as $j\to\infty$.
\end{proposition}

\begin{proof}
For each $n\in\N$ let $S_n\in C_p(\mu,k,R)$ be arbitrary, and get $\{n_j\}_{j\in\N}$ and $S\subseteq R$ with $\#S\le k$ as in  the first part of Proposition~\ref{prop:structural}.
In order to prove the first claim of this result, it remains to show that $S$ is optimal and that we have $S_{n_j}\to S$ in $\tau$.
We easily see that $S$ is feasible by construction.
To show optimality, we let $J := \{j_i\}_{i\in\N}$ be an arbitrary subsequence, and we use the inequality of Proposition~\ref{prop:structural} along with Lemma~\ref{lem:partial-cty} to get
\begin{equation}\label{eqn:second-1}
	\begin{split}
		W_p(S,\mu) &\le \int_{\mathcal{X}}\liminf_{i\to\infty}d^p(y,S_{n_{j_i}})\, \diff\mu(y) \\
		&\le \liminf_{i\to\infty}\int_{\mathcal{X}}d^p(y,S_{n_{j_i}})\, \diff\mu(y) \\
		&= \liminf_{i\to\infty}W_p(S_{n_{j_i}},\mu) = \liminf_{i\to\infty}W_p(S_{n_{j_i}},\mu_{n_{j_i}})
	\end{split}
\end{equation}
Now take arbitrary $S'\subseteq R$ with $\#S'= k$.
By Lemma~\ref{lem:Kur-lower-subseq} we can get a subsequence $\{i_u\}_{u\in\N}$ and a set $S_u'\subseteq R_{n_{j_{i_u}}}$ with $\#S_u' =\#S' = k$ for all $u\in\N$ with $\dHaus(S'_u,S')\to 0$.
Therefore, Lemma~\ref{lem:W-cts} implies
\begin{align*}
	\liminf_{i\to\infty}W_p(S_{n_{j_i}},\mu_{n_{j_i}}) &\le \liminf_{u\to\infty}W_p(S_{n_{j_{i_u}}},\mu_{n_{j_{i_u}}}) \\
	&\le \liminf_{u\to\infty}W_p(S_{u}',\mu_{n_{j_{i_u}}}) = W_p(S',\mu).
\end{align*}
Taking the infimum over all such $S'$, we have proven
\begin{equation*}
	W_p(S,\mu) \le m_{k,p}(\mu,R).
\end{equation*}
Therefore, $S$ is optimal, hence $S\in C_p(\mu,k,R)$.
By the non-singularity of $(\mu,k,R)$, this also implies $\#S=k$.
Therefore, combining Proposition~\ref{prop:bdd} with the second part of Proposition~\ref{prop:structural}, we get $S_{n_j}\to S$ in $\tau$.
In other words, we can write $S_{n_j} = \{a_1^{j},\ldots, a_{k}^{j}\}$ for all $j\in\N$ and $S = \{a_1,\ldots, a_{k}\}$ so that we have $a_{\ell}^{j}\to a_{\ell}$ in $\tau$ for all $1\le \ell \le k$.

Before we can prove the second claim, we must make some preparations.
Let us define, for each subsequence $J:= \{j_i\}_{i\in\N}$, the set $B_{J} := \{y\in \mathcal{X}: d(y,S) = \liminf_{i\to\infty}d(y,S_{n_{j_i}})\}$.
Crucially, observe that the optimality of $S$ implies that the inequalities of \eqref{eqn:second-1} are in fact equalities, and hence that $\mu(B_{J}) = 1$ for each $J$.
While of course the uncountable intersection of full-measure sets need not have full measure, it suffices by Lemma~A.1 in the supplementary material to further establish that each set is $d$-closed.
To see that $B_J$ is indeed $d$-closed, suppose that $\{y_m\}_{m\in\N}$ in $B_J$ have $y_m\to y\in \mathcal{X}$ in $d$.
Then:
\begin{align*}
	&\left|d(y,S) - \liminf_{i\to\infty}d(y,S_{n_{j_i}})\right| \\
	&= |d(y,S) - d(y_m,S)| + \left|\liminf_{i\to\infty}d(y_m,S_{n_{j_i}}) - \liminf_{i\to\infty}d(y,S_{n_{j_i}})\right| \\
	&= |d(y,S) - d(y_m,S)| + \limsup_{i\to\infty}|d(y_m,S_{n_{j_i}}) - d(y,S_{n_{j_i}})| \le 2d(y_m,y) \to 0.
\end{align*}
Consequently, the set $B := \bigcap_{J}B_J$ satisfies $\mu(B)= 1$.

Next, we define the sets $V_{\ell} := \{y\in \mathcal{X}: d(y,a_{\ell}) = d(y,S)\}$ for $1\le\ell\le k$.
We claim that for all $1\le\ell\le k$ there exists $y\in V_{\ell}\cap B$ satisfying $d(y,S_{n_j}) = d(y,a_{\ell}^{j})$ for sufficiently large $j\in\N$.
If not, then there exists $1\le \ell\le k$ such that for all $y\in V_{\ell}\cap B$ there exists $\ell(y)\in \{1,\ldots, k\}\setminus\{\ell\}$ such that we have $d(y,S_{n_j}) = d(y,a_{\ell(y)}^{j})$ for infinitely many $j\in\N$.
Denoting by $\{j_i\}_{i\in\N}$ such a subsequence, we can use (W2) and $y\in V_{\ell}\cap B$ to bound:
\begin{align*}
	d\left(y,S\setminus\{a_{\ell}\}\right) \le d\left(y,a_{\ell(y)}\right) &\le \liminf_{j\to\infty}d\left(y,a_{\ell(y)}^{j}\right) \\
	&\le \liminf_{i\to\infty}d\left(y,a_{\ell(y)}^{j_i}\right) \\
	&= \liminf_{i\to\infty}d\left(y,S_{n_{j_i}}\right) = d(y,S) = d(y,a_{\ell}).
\end{align*}
In other words, $V_{\ell}\cap B\subseteq \bigcup_{\ell'\neq \ell}V_{\ell'}$.
However, this implies that $S\setminus\{a_{\ell}\}$ is optimal, which contradicts the non-singularity of $(\mu,k,R)$.

Now we put all the pieces together.
For each $1\le \ell \le k$, we use the above to get some $y\in V_{\ell}\cap B$ satisfying $d(y,S_{n_j}) \in d(y,a_{\ell}^{j})$ for sufficiently large $j\in\N$.
Then let $\{j_i\}_{i\in\N}$ be any subsequence satisfying
\begin{equation*}
	\limsup_{j\to\infty}d\left(y,S_{n_j}\right) = \lim_{i\to\infty}d\left(y,S_{n_{j_i}}\right),
\end{equation*}
and use (W2) and $y\in V_{\ell}\cap B$ to get:
\begin{align*}
	\limsup_{j\to\infty}d\left(y,a_{\ell}^{j}\right) =\limsup_{j\to\infty}d\left(y,S_{n_j}\right) &= \lim_{i\to\infty}d\left(y,S_{n_{j_i}}\right) \\
	&= \liminf_{i\to\infty}d\left(y,S_{n_{j_i}}\right) \\
	&= d(y,S) = d(y,a_{\ell}) \le \liminf_{j\to\infty}d\left(y,a_{\ell}^{j}\right).
\end{align*}
This shows $d(y,a_{\ell}^{j})\to d(y,a_{\ell})$ as $j\to\infty$.
Combining this with $a_{\ell}^{j}\to a_{\ell}$ in $\tau$ as $j\to\infty$ and (W3), this implies $a_{\ell}^{j}\to a_{\ell}$ in $d$ as $j\to\infty$.
Thus we have shown $S_{n_j}\to S$ in $d$, and this completes the proof.
\end{proof}

The next goal is to show that Proposition~\ref{prop:cpt-operator} implies a uniform convergence result with respect to the Hausdorff metric $\dHaus$.
In order to make this precise, we need the following form of regularity.

\begin{lemma}\label{lem:C-regularity}
	If $(\mathcal{X},d)$ is separable, $\tau$ satisfies \textnormal{(W1)}, \textnormal{(W2)}, and \textnormal{(W3)}, and $(\mu,k,R)\in\mathcal{P}_p(\mathcal{X})\times\N\times\closed^{\tau}(\mathcal{X})$ is non-singular, then $C_p(\mu,k,R)$ is non-empty and $\dHaus$-compact.
\end{lemma}

\begin{proof}
Lemma~\ref{lem:non-empty} gives non-emptiness, and compactness follows immediately from the second part of Proposition~\ref{prop:cpt-operator}.
\end{proof}

\begin{lemma}\label{lem:nonsingular-limits}
	If $(\mathcal{X},d)$ is separable, $\tau$ satisfies \textnormal{(W1)}, \textnormal{(W2)}, and \textnormal{(W3)}, and $\{(\mu_n,R_n)\}_{n\in\N}$  in $\mathcal{P}_p(\mathcal{X})\times\closed^{\tau}(\mathcal{X})$ and $(\mu,k,R)$ in $\mathcal{P}_p(\mathcal{X})\times \N\times\closed^{\tau}(\mathcal{X})$ are such that $(\mu,k,R)$ is non-singular and
	\begin{itemize}
		\item $\mu_n\to \mu$ in $\weakp$, and
		\item $\kurouter^{\tau}_{n\in\N}R_n \subseteq R \subseteq \kurinner^{d}_{n\in\N}R_n$,
	\end{itemize}
	then $m_{\ell,p}(\mu_n,R_n)\to m_{\ell,p}(\mu,R)$ for all $1\le \ell\le k$.
\end{lemma}

\begin{proof}
	Let us show that any subsequence of $\{m_{\ell,p}(\mu_{n},R_{n})\}_{n\in\N}$ has a further subsequence converging to $m_{\ell,p}(\mu,R)$.
	Indeed, take arbitrary $\{n_j\}_{j\in\N}$, and, for each $j\in\N$ use Lemma~\ref{lem:non-empty} to get some $S_j\in C_{p}(\mu_{n_j},\ell,R_{n_j})$.
	Since $(\mu,k,R)$ being non-singular certainly implies that $(\mu,\ell,R)$ is non-singular, we can apply Proposition~\ref{prop:cpt-operator} to get a further subsequence $\{j_i\}_{i\in\N}$ and some $S\in C_p(\mu,\ell,R)$ such that $\dHaus(S_{j_i}, S)\to 0$ as $i\to\infty$.
	Finally, note that Lemma~\ref{lem:W-cts} gives
	\begin{align*}
	m_{\ell,p}(S_{j_i},R_{n_{j_i}}) = W_p(S_{j_i},\mu_{n_{j_i}}) \to W_p(S,\mu) = m_{\ell,p}(\mu,R)
	\end{align*}
	whence the result.
\end{proof}

This result can also be used to show the following sufficient condition for non-singularity.
(However, note that in applications we are often able to show non-singularity of $(\bar \mu_n,k_n,R_n)$ for all $n\in\N$ directly, for example if the population distribution is atomless.)

\begin{lemma}\label{lem:nonsingular-inherited}
	If $(\mathcal{X},d)$ is separable, $\tau$ satisfies \textnormal{(W1)}, \textnormal{(W2)}, and \textnormal{(W3)}, and $\{(\mu_n,R_n)\}_{n\in\N}$  in $\mathcal{P}_p(\mathcal{X})\times\closed^{\tau}(\mathcal{X})$ and $(\mu,k,R)$ in $\mathcal{P}_p(\mathcal{X})\times \N\times\closed^{\tau}(\mathcal{X})$ are such that $(\mu,k,R)$ is non-singular and
	\begin{itemize}
		\item $\mu_n\to \mu$ in $\weakp$, and
		\item $\kurouter^{\tau}_{n\in\N}R_n \subseteq R \subseteq \kurinner^{d}_{n\in\N}R_n$,
	\end{itemize}
	then $(\mu_n,k_n,R_n)$ is non-singular for sufficiently large $n\in\N$.
\end{lemma}

\begin{proof}
Set $\varepsilon:= \min\{m_{\ell-1,p}(\mu,R) - m_{\ell,p}(\mu,R): 1 < \ell \le k\}$ which is strictly positive since $(\mu,k,R)$ is non-singular.
By Lemma~\ref{lem:nonsingular-limits} we can take $n\in\N$ sufficiently large to get $\sup_{1<\ell\le k}|m_{\ell,p}(\mu_n,R_n) - m_{\ell,p}(\mu,R)| < \varepsilon/2$, and for such $n$ we have $m_{1,p}(\mu_n,R_n) > \cdots > m_{k,p}(\mu_n,R_n)$.
\end{proof}

Now our uniform convergence result follows:

\begin{proposition}\label{prop:clustering-cty}
	If $(\mathcal{X},d)$ is separable, $\tau$ satisfies \textnormal{(W1)}, \textnormal{(W2)}, and \textnormal{(W3)}, and $\{(\mu_n,k_n,R_n)\}_{n\in\N}$ and $(\mu,k,R)$ in $\mathcal{P}_p(\mathcal{X})\times \N\times \closed^{\tau}(\mathcal{X})$ are such that $(\mu,k,R)$ is non-singular, and
	\begin{itemize}
		\item $\mu_n\to \mu$ in $\weakp$,
		\item $k_n\neq k$ for finitely many $n\in\N$, and
		\item $\kurouter^{\tau}_{n\in\N}R_n \subseteq R \subseteq \kurinner^{d}_{n\in\N}R_n$,
	\end{itemize}
	then we have
	\begin{equation*}
		\max_{S_n\in C_p(\mu_n,k_n,R_n)}\min_{S\in C_p(\mu,k,R)}\dHaus(S_n,S)\to 0
	\end{equation*}
	as $n\to\infty$.
\end{proposition}

\begin{proof}
	Lemma~\ref{lem:nonsingular-inherited} and Lemma~\ref{lem:C-regularity} imply that $\{C_p(\mu_n,k_n,R_n)\}_{n\in\N}$ and $C_p(\mu,k,R)$ are in $\cbdd(\cbdd(\mathcal{X}),\dHaus)$ for sufficiently large $n\in\N$, so the supremum and infimum can indeed be replaced with maximum and minimum.
	(See Remark~\ref{rem:dvecHaus-finite}.)
	Now to prove the claim it suffice to show that for each subsequence $\{n_j\}_{j\in\N}$ there exists a further subsequence $\{j_i\}_{i\in\N}$ satisfying
	\begin{equation*}
		\max_{S_i\in C_p(\mu_{n_{j_i}},k_{n_{j_i}},R_{n_{j_i}})}\min_{S\in C_p(\mu,k,R)}\dHaus(S_i,S)\to 0
	\end{equation*}
	as $i\to\infty$.
	To do this, take arbitrary $\{n_j\}_{j\in\N}$, use Lemma~\ref{lem:C-regularity} to get $S_j'\in C_p(\mu_{n_j},k_{n_j},R_{n_j})$ satisfying
	\begin{equation*}
		\min_{S\in C_p(\mu,k,R)}\dHaus(S_j',S) = \max_{S_j\in C_p(\mu_{n_j},k_{n_j},R_{n_j})}\min_{S\in C_p(\mu,k,R)}\dHaus(S_j,S).
	\end{equation*}
	Finally, use Proposition~\ref{prop:cpt-operator} to get $S'\in C_p(\mu,k,R)$ and $\{j_i\}_{i\in\N}$ such that we have $\dHaus(S_{j_i}',S')\to 0$ as $i\to\infty$, and note that this implies
	\begin{align*}
		\max_{S_i\in C_p(\mu_{n_{j_i}},k_{n_j},R_{n_j})}\min_{S\in C_p(\mu,k,R)}\dHaus(S_{j_i}',S) &= \min_{S\in C_p(\mu,k,R)}\dHaus(S_{j_i}',S) \\
		&\le \dHaus(S_{j_i}',S') \to 0,
	\end{align*}
	as needed.
\end{proof}

In addition to the preceding continuity result, we describe the continuity of the adaptive choice of $k$ arising in the elbow method.
(Note that in this case, our ``continuity'' statement is a bona fide continuity statement.)
To set this up, we let $\mu\in \mathcal{P}_p(\mathcal{X})$ be arbitrary, and we define
\begin{equation*}
	\Delta^2m_{k,p}(\mu):=m_{k+1,p}(\mu)+m_{k-1,p}(\mu)-2m_{k,p}(\mu)
\end{equation*}
for $k\ge 2$; the restriction to $k\ge 2$ is equivalent to adopting the convention $m_{0,p}(\mu) = \infty$.
Then we define the function $\kelbow_p(\mu):\mathcal{P}_p(\mathcal{X})\to \N\cup\{\infty\}$ via
\begin{equation*}
	\kelbow_p(\mu) := \min\{\arg\max\{\Delta^2m_{k,p}(\mu): k\in\N \}\}
\end{equation*}
for $\mu\in \mathcal{P}_p(\mathcal{X})$.
This leads us to the following result; the proof is straightforward but slightly long, so we defer it to the supplementary material.

\begin{proposition}\label{prop:elbow-k-cts}
	Suppose that $(\mathcal{X},d)$ is a Polish metric space and that $\mu\in \mathcal{P}_p(\mathcal{X})$ satisfies $\#\supp(\mu) = \infty$ and $\#\arg\max\{\Delta^2m_{k,p}(\mu): k\in\N , k\ge 2\} =1$.
	Then, the function $\kelbow_p:\mathcal{P}_p(\mathcal{X})\to \N$ is continuous at $\mu$.
\end{proposition}

The condition $\#\arg\max\{\Delta^2m_{k,p}(\mu): k\in\N , k\ge 2\} =1$ should be interpreted as saying that the distribution $\mu$ has a uniquely-defined number of clusters, in the sense of the elbow method.
Thus, the preceding result says that, if $\mu_n\to \mu$ in $\weak^p$ and $\mu$ has a uniquely-defined number of clusters, then the same is true for $\mu_n$ for sufficiently large $n\in\N$.

\section{Probabilistic results}\label{sec:results}

We now show how the considerations of the previous section can be used to prove a number of limit theorems for clustering procedures of interest in statistical applications.
More precisely, Subsection~\ref{subse:results-SLLN} addresses the main question of strong consistency of adaptive clustering procedures under IID data.
We then consider two secondary questions.
That is, in Subsection~\ref{subse:results-MCs} we address the question of strong consistency for $(k,p)$-means and $(k,p)$-medoids under MC data, and in Subsection~\ref{subse:results-LDP} we address the question of large deviations estimates for $(k,p)$-means under IID data.
\\

For this section, we let $(\mathcal{X},d)$ be a separable metric space admitting a weak convergence, and $1\le p < \infty$.
An underlying probability space will be denoted $(\Omega,\F,\P)$ and will be assumed to be complete.

\begin{remark}
	One may be concerned about the measurability of the ``events'' appearing throughout this section.
	As the proofs all demonstrate, these events are supersets of bona fide events whose $\P$-probabilities are equal to one, hence they are at least measurable with respect to the $\P$-completion of $\F$.
	For this reason we will not address further questions of measurability in this paper; the interested reader should see \cite[Section~3]{EvansJaffeFrechet} for a resolution of similar measurability concerns in the context of Fr\'echet means, which is equivalent to fixing $k=1$ in the current work.
\end{remark}

\subsection{Strong consistency  for IID data}\label{subse:results-SLLN}

Suppose that $Y_1,Y_2,\ldots$ is an IID sequence of random variables with common distribution $\mu\in \mathcal{P}_p(\mathcal{X})$, and define the empirical measures via $\bar \mu_n := \frac{1}{n}\sum_{i=1}^{n}\delta_{Y_i}$ for all $n\in\N$.
First, we have the following fundamental strong consistency result, which follows easily from our preparations.
\begin{proof}[Proof of Theorem~\ref{thm:Haus-consistency}]
	We have $\bar \mu_n\to \mu$ in $\weak^p$ almost surely from \cite[Proposition~4.1]{EvansJaffeFrechet}, so the result follows from Proposition~\ref{prop:clustering-cty}.
\end{proof}

To see that this result has immediate consequences for strong consistency of $(k,p)$-means, suppose that $\mu\in \mathcal{P}_p(\mathcal{X})$ has $\supp(\mu) = \infty$.
Then fix $k\in\N$, and take $k_n= k$ and $R_n = R = \mathcal{X}$ for all $n\in\N$.
It immediately follows that we have
\begin{equation*}
\max_{S_n\in C_{k,p}(\bar \mu_n)}\min_{S\in C_{k,p}(\mu)}\dHaus(S_n,S)\to 0
\end{equation*}
almost surely.
In words, this says that $(k,p)$-means is strongly consistent.

For $(k,p)$-medoids, we fix $k_n = k\in\N$ for $n\in\N$, and we take $R_n = \supp^{\tau}(\bar \mu_n) = \{Y_1,\ldots, Y_n\}$ for $n\in\N$ and $R = \supp^{\tau}(\mu)$.
In order to apply Theorem~\ref{thm:Haus-consistency},we require the following:

\begin{lemma}\label{lem:supp-converge}
	If $(\mathcal{X},d)$ is a separable and $\tau$ is any Hausdorff topology weaker than $d$, then we have $\kurouter^{\tau}_{n\in\N}\supp^{\tau}(\bar \mu_n)\subseteq \supp^{\tau}(\mu)\subseteq \kurinner^{d}_{n\in\N}\supp^{\tau}(\bar \mu_n)$ almost surely.
\end{lemma}

\begin{proof}
	Recall that $\supp^{\tau}(\bar \mu_n) = \{Y_1,\ldots, Y_n\}$ almost surely.
	To show $\kurouter^{\tau}_{n\in\N}\supp^{\tau}(\bar \mu_n)\subseteq \supp^{\tau}(\mu)$ almost surely, consider any $x\notin \supp^{\tau}(\mu)$.
	This means there exists a $\tau$-open set $U\subseteq \mathcal{X}$ with $x\in U$ with $\mu(U) = 0$.
	In particular, this implies $Y_n\notin U$ for all $n\in\N$ almost surely, which implies that $x$ cannot be the limit point of any $\tau$-convergent subsequence of elements of $\{\supp^{\tau}(\bar \mu_n)\}_{n\in\N}$.
	To show $\supp^{\tau}(\mu)\subseteq \kurinner^{d}_{n\in\N}\supp^{\tau}(\bar \mu_n)$ almost surely, note that we almost surely have $\bar \mu_n\to \mu$ topology of weak convergence with respect to $d$.
	Thus, the result follows from \cite[Lemma~2.15]{EvansJaffeFrechet}.
\end{proof}

Consequently, we get
\begin{equation*}
	\max_{S_n\in \clustermedoid_{k,p}(\bar \mu_n)}\min_{S\in \clustermedoid_{k,p}(\mu)}\dHaus(S_n,S)\to 0
\end{equation*}
almost surely, which states that $(k,p)$-medoids is strongly consistent.
Note that we used $\supp^{\tau}(\mu)$ instead of $\supp(\mu)$ in this result, and that we of course have $\supp^{\tau}(\mu)\supseteq\supp(\mu)$ in general.

For $(k,p)$-means where $k$ is chosen adaptively according to the elbow method, we take $k_n = \kelbow_p(\bar \mu_n)$ for all $n\in\N$ as well as $k=\kelbow_p(\mu)$, and $R_n = R =\mathcal{X}$ for all $n\in\N$.
If we assume additionally that $(\mathcal{X},d)$ is a Polish metric space and $\#\arg\max\{\Delta^2m_{k,p}(\mu): k\in\N, k\ge 2\} =1$, then Proposition~\ref{prop:elbow-k-cts} implies
\begin{equation*}
	\max_{S_n\in \clusterelbow_{p}(\bar \mu_n)}\min_{S\in \clusterelbow_{p}(\mu)}\dHaus(S_n,S)\to 0
\end{equation*}
almost surely.
In words, $(k,p)$-means, where $k$ is chosen adaptively according to the elbow method, is strongly consistent provided that $\mu$ has a uniquely-defined number of clusters in the sense of the elbow method.

Let us also address a computational application of these results.
Indeed, if $\mu\in\mathcal{P}_p(\mathcal{X})$ satisfies $\supp(\mu) = \mathcal{X}$, then for any $k\in\N$ we have $\clustermedoid_{k,p}(\mu)=C_{k,p}(\mu)$, hence
\begin{equation*}
	\max_{S_n\in \clustermedoid_{k,p}(\bar \mu_n)}\min_{S\in C_{k,p}(\mu)}\dHaus(S_n,S)\to 0
\end{equation*}
almost surely.
Roughly speaking, this means that each set of empirical $(k,p)$-medoids cluster centers must be close to some set of population $(k,p)$-means cluster centers, at least asymptotically.
To see that this is interesting, recall on the one hand that there is no known algorithm to exactly compute the empirical $(k,p)$-means cluster centers of $n$ data points, and observe on the other hand that the empirical $(k,p)$-medoids cluster centers of $n$ data points can be computed exactly in $O(Dkn^{k+1})$ time, where $D$ is the time of computing the distance between any two points.
Thus, it is interesting to notice that the empirical $(k,p)$-medoids problem can be consistently used as a surrogate for the harder $(k,p)$-means problem.
Although the cost of exactly computing $(k,p)$-medoids is certaintly large, we believe that it would be interesting to try to develop efficient randomized algorithms for this task, in the spirit of \cite{UltraFastMedoids}.

\subsection{Strong consistency for MC data}\label{subse:results-MCs}

Suppose that $Y_1,Y_2,\ldots$ is an aperiodic Harris-recurrent Markov chain (MC) on $\mathcal{X}$, and let $\mu$ denote its unique stationary distribution.
As always, write $\bar \mu_n := \frac{1}{n}\sum_{i=1}^{n}\delta_{Y_i}$ for the empirical measures of the first $n\in\N$ data points.
Then we get the following:

\begin{proof}[Proof of Theorem~\ref{thm:MC-consistency}]
By \cite[Theorem~4]{MCMC} we have $\bar \mu_n \to \mu$ in total variation, and by \cite[Fact~5]{MCMC} we have $\frac{1}{n}\sum_{i=1}^{n}d^p(x,Y_i)\to \int_{\mathcal{X}}d^p(x,y)\,\diff\mu(y)$, both holding almost surely; in particular, we have $\bar \mu_n\to \mu$ in $\weakp$ almost surely.
Thus, the result follows from Theorem~\ref{thm:Haus-consistency} by taking $R_n = R = \mathcal{X}$ and $k_n = k$ for all $n\in\N$.
\end{proof}

In words, this result guarantees that $(k,p)$-means is strongly consistent when applied to data coming from a suitable MC.
A typical application of interest in which data assumed to follow such a Markovian structure is the setting where $Y_1,Y_2,\ldots$ are samples from a posterior distribution in Bayesian statistics which are computed via Markov chain Monte Carlo (MCMC).
In this setting, one typically has $\mathcal{X}=\R^m$ with its usual metric, which certainly satisfies our hypothesis.
Moreover, if $\mu$ and $P(x,\cdot)$ for all $x\in\R^m$ have strictly positive densities with respect to the Lebesgue measure, then the only hypothesis that one needs to check is the simple $\int_{\R^m}|y|^p\, \diff\mu(y) < \infty$.

The case of $(k,p)$-medoids applied to data coming from a MC is more subtle.
For the sake of simplicity, we focus on the case of MCs on finite state spaces.
Thus, let us make the following assumption for the remainder of this subsection: $(\mathcal{X},d)$ is a finite metric space, and $\mathcal{X}$ can be decomposed into $\mathcal{X} = \mathcal{X}_0\sqcup\mathcal{X}_1$ where $\mathcal{X}_0$ are the inessential states with respect to $P$ (that is, the states that are almost surely visited finitely often) and $\mathcal{X}_1$ are the essential states with respect to $P$ (that is, the states that are almost surely visited infinitely often), and $P$ is aperiodic and irreducible on $\mathcal{X}_1$.

It is illustrative to see what can go wrong in the simplest possible setting:

\begin{example}
	Consider $\mathcal{X} =  \{-1,0,1\}$ with the metric inherited from the real line.
	Then let $Y_1,Y_2,\ldots$ be a MC with $Y_1 = 0$ and with the transition matrix
	\begin{equation*}
		P = \begin{pmatrix}
			\sfrac{1}{2} & 0 & \sfrac{1}{2} \\
			\sfrac{1}{3} & \sfrac{1}{3} & \sfrac{1}{3} \\
			\sfrac{1}{2} & 0 & \sfrac{1}{2}
		\end{pmatrix}.
	\end{equation*}
	In words, this MC stays at state 0 for a geometric amount of time, then subsequently visits $\{-1,1\}$ independently and uniformly at random.
	It is clear that the only stationary distribution for this MC is $\mu= \frac{1}{2}\delta_{-1}+\frac{1}{2}\delta_{1}$, and that it is aperiodic and Harris-recurrent.
	
	Now write $I:= \max\{i\in\N: Y_i = 0 \}$ for the number of data equal to 0 and $Z_n := \sum_{i=1}^{n}Y_i$ for the cumulative sum of the data.
	Observe in particular that $I < \infty$ almost surely and that we have $\bar \mu_n = \frac{1}{2n}(n-Z_n-I)\delta_{-1} + \frac{I}{n}\delta_{0} +\frac{1}{2n}(n+Z_n-I)\delta_{1}$ for $n> I$.
	Moreover, $\{Z_n\}_{n> I}$ is a simple symmetric random walk.
	Next notice that on $\{n> I,Z_n = 0\}$ we have
	\begin{equation*}
			\int_{\mathcal{X}}d^p(\pm 1,y)\, d\bar \mu_n(y) = \frac{I}{n}+\left(1-\frac{I}{n}\right)2^{p-1} > 1-\frac{I}{n} = \int_{\mathcal{X}}d^p(0,y)\, d\bar \mu_n(y),
	\end{equation*}
	hence $\clustermedoid_{1,p}(\bar \mu_n) = \{\{0\}\}$.
	Also, $\clustermedoid_{1,p}(\mu) = \{\{-1\},\{1\}\}$.
	In particular, we have shown
	\begin{equation*}
		\max_{S_n\in \clustermedoid_{1,p}(\bar \mu_n)}\min_{S\in \clustermedoid_{1,p}(\mu)}\dHaus(S_n,S) = 1
	\end{equation*}
	on $\{n> I,Z_n = 0\}$.
	Finally, notice that we have $\P(n> I,Z_n = 0 \text{ for infinitely many } n\in\N) = 1$ by the recurrence of the simple random walk, hence we are able to conclude that $\max_{S_n\in \clustermedoid_{1,p}(\bar \mu_n)}\min_{S\in \clustermedoid_{1,p}(\mu)}\dHaus(S_n,S) \to 0$ occurs with probability zero.
	In words, $(k,p)$-medoids for data coming from this MC is strongly inconsistent.
\end{example}

We now introduce a method to repair this apparent deficiency.
The difficulty in the preceding example is that the adaptively-chosen domain of the cluster centers includes some states not included in the support of the stationary distribution, so, to get around this, we allow our clustering procedure to ``forget'' some initial segment of states.
This is similar to giving the MC a suitable ``burn-in'' period.

Indeed, let $f = \{f_n\}_{n\in\N}$ be any integer sequence with $0\le f_n \le n$ for all $n\in\N$.
Then define
\begin{equation*}
	\bar \mu_n^f := \frac{1}{n-f_n}\sum_{i=f_n+1}^{n}\delta_{Y_i}
\end{equation*}
for $n\in\N$; that is, $\{\bar \mu_n^f\}_{n\in\N}$ are the empirical measures of only the most recent data points, where we forget initial segments of sizes determined by $f$.

\begin{lemma}\label{lem:forgetting}
	In the setting of Theorem~\ref{thm:MC-consistency}, if $f_n/n\to 0$ and $f_n\to \infty$, then
	\begin{itemize}
		\item[(i)] $\bar \mu_n^f \to \mu$ in $\weak$ almost surely, and
		\item[(ii)] $\kurlimit_{n\in\N}\supp(\bar \mu_n^f) =\supp(\mu)$ almost surely.
	\end{itemize}
\end{lemma}

\begin{proof}
	For (i), note that $\mathcal{X}$ being finite means that convergence in $\weak$ is equivalent to convergence in the total variation norm, $\|\cdot\|_{\textrm{TV}}$.
	To use this, write
	\begin{equation*}
		\bar \mu_n^f = \frac{n}{n-f_n}\cdot\frac{1}{n}\sum_{i=1}^{n}\delta_{Y_i} - \frac{1}{n-f_n}\sum_{i=1}^{f_n}\delta_{Y_i}
	\end{equation*}
	hence
	\begin{equation*}
		\begin{split}
			\|\bar \mu_n^f-\bar \mu_n\|_{\textrm{TV}} &\le \left|\frac{n}{n-f_n}-1\right|\frac{1}{n}\sum_{i=1}^{n}\|\delta_{Y_i}\|_{\textrm{TV}} + \frac{1}{n-f_n}\sum_{i=1}^{f_n}\|\delta_{Y_i}\|_{\textrm{TV}} \\
			&\le \left|\frac{n}{n-f_n}-1\right| + \frac{f_n}{n-f_n}.
		\end{split}
	\end{equation*}
	Now note that $f_n/n\to 0$ implies that the right side goes to zero, hence $\|\bar \mu_n^f-\bar \mu_n\|_{\textrm{TV}}\to 0$.
	Since $\|\bar \mu_n-\mu\|_{\textrm{TV}}\to 0$ almost surely by the classical ergodic theorem, we conclude $\|\bar \mu_n^f-\mu\|_{\textrm{TV}}\to 0$ whence (i).
	For (ii), note by (i) and \cite[Lemma~2.10]{EvansJaffeFrechet} that we have $\supp(\mu) \subseteq \kurinner_{n\in\N}\supp(\bar \mu_n^f)$ almost surely.
	For the converse, suppose that $x\in\kurouter_{n\in\N}\supp(\bar \mu_n^f)$.
	Since $X$ is discrete, this means that there is a subsequence $\{n_j\}_{j\in\N}$ with $x\in\supp(\bar \mu_{n_j}^f)$ for all $j\in\N$.
	Consequently, there is some sequence $\{\ell_j\}_{j\in\N}$ (not necessarily non-decreasing) with $f_{n_j}+1\le \ell_j \le  n_j$ and $Y_{\ell_j} = x$ for all $j\in\N$.
	Since $f_{n}\to \infty$, this means  $Y_1,Y_2,\ldots$ visits $x$ infinitely often.
	But $\mathcal{X}_0$ is inessential with respect to $P$, so we must have $x\notin \mathcal{X}_0$.
	This implies $x\in \mathcal{X}_1$ hence $x\in \supp(\mu)$ since $\mathcal{X}_1$ is irreducible and aperiodic with respect to $P$.
	We have shown $\kurouter_{n\in\N}\supp(\bar \mu_n^f)\subseteq \supp(\mu)$ almost surely, so combining with the first part gives (ii).
\end{proof}

This, in particular, gives the following strong consistency.

\begin{theorem}\label{thm:MC-medoids}
	In the setting of Theorem~\ref{thm:MC-consistency}, if $f_n/n\to 0$ and $f_n\to \infty$, then we have
	\begin{equation*}
		\max_{S_n\in \clustermedoid_{k,p}(\bar \mu_n^f)}\min_{S\in \clustermedoid_{k,p}(\mu)}\dHaus(S_n,S)\to 0
	\end{equation*}
	almost surely.
\end{theorem}

\begin{proof}
	Immediate by Theorem~\ref{thm:Haus-consistency} and Lemma~\ref{lem:forgetting}.
\end{proof}

To be concrete in the setting above, one can take $f_n := \lfloor \log n\rfloor$ for $n\in\N$.

\subsection{Large deviations for IID data}\label{subse:results-LDP}

Suppose that $Y_1,Y_2,\ldots$ is an IID sequence of $\mathcal{X}$-valued random variables with common distribution $\mu$.
As always, let us define the empirical measures via $\bar \mu_n := \frac{1}{n}\sum_{i=1}^{n}\delta_{Y_i}$ for all $n\in\N$.
We have already established the almost sure convergence
\begin{equation*}
\max_{S_n\in C_{k,p}(\bar \mu_n)}\min_{S\in C_{k,p}(\mu)}\dHaus(S_n,S)\to 0,
\end{equation*}
and this of course implies the convergence in probability
\begin{equation*}
\P\left(\sup_{S_n\in C_{k,p}(\bar \mu_n)}\inf_{S\in C_{k,p}(\mu)}\dHaus(S_n,S)\ge \varepsilon\right)\to 0,
\end{equation*}
for all $\varepsilon > 0$.
Presently, we address the question of determining the rate at which these probabilities decay to zero.
This is the domain of large deviations theory, the basics of which can be found in \cite{DZ}.

\begin{remark}\label{rem:sup-inf-achieved}
	It is shown in Proposition~\ref{prop:clustering-cty} that, for sufficiently large $n\in\N$, the supremum and infimum in the stated uniform convergence can be replaced with maximum and minimum.
	However, a careful inspection of the proof reveals (via Lemma~\ref{lem:C-regularity} and
	Lemma~\ref{lem:nonsingular-inherited}) that the requisite size of $n\in\N$ is actually random.
	Thus, statements concerning convergence in probability, like the above, must use supremum and infimum.
\end{remark}

To perform this analysis, let us define the map $H(\cdot|\mu): \mathcal{P}(\mathcal{X})\to [0,\infty]$ via
\begin{equation*}
	H(\nu|\mu)  = \begin{cases}
		\int_{\mathcal{X}}\frac{d\nu}{\diff\mu} \log\left(\frac{d\nu}{\diff\mu}\right)\diff\mu, &\text{ if } \nu \textrm{ is absolutely continuous with respect to } \mu, \\
		\infty, &\text{ if } \nu \textrm{ is not absolutely continuous with respect to } \mu, \\
	\end{cases}
\end{equation*}
for $\nu\in \mathcal{P}(\mathcal{X})$.
We call $H(\nu|\mu)$ the \textit{KL divergence} from $\mu$ to $\nu$.
The key to proving our large deviations upper bound is applying a type of contraction to an existing large deviations principle for the empirical measures $\{\bar  \mu_n\}_{n\in\N}$ in the space $(\mathcal{P}_p(\mathcal{X}),\weak^p)$.

\begin{proof}[Proof of Theorem~\ref{thm:LDP}]
	Consider the set
	\begin{equation*}
		\begin{split}
		A &:=  \left\{\nu\in \mathcal{P}_p(\mathcal{X}): \sup_{T\in C_{k,p}(\nu)}\inf_{S\in C_{k,p}(\mu)}\dHaus(T,S) \ge \varepsilon \right\} \\
		&= \left\{\nu\in \mathcal{P}_p(\mathcal{X}): \vec{D}_{\textnormal{H}}(C_{k,p}(\nu),C_{k,p}(\mu)) \ge \varepsilon \right\}.
		\end{split}
	\end{equation*}
	Here, we write $\vec{D}_{\textnormal{H}}$ for the one-sided Hausdorff distance on the collection of all non-empty $\dHaus$-closed subsets of $\cbdd(\mathcal{X})$.
	(See Remark~\ref{rem:dvecHaus-finite} and Remark~\ref{rem:sup-inf-achieved}.)
	Now let us show that $A$ is $\weak^p$-closed.
	Indeed, suppose $\{\nu_n\}_{n\in\N}$ in $A$ and $\nu\in \mathcal{P}_p(\mathcal{X})$ have  $\nu_n\to \nu$ in $\weak^p$.
	Then apply the non-symmetric triangle inequality \eqref{dvecHaus-triangle-ineq} to get:
	\begin{equation*}
	\vec{D}_{\textnormal{H}}(C_{k,p}(\nu_n),C_{k,p}(\nu)) +  \vec{D}_{\textnormal{H}}(C_{k,p}(\nu),C_{k,p}(\mu)) \ge \vec{D}_{\textnormal{H}}(C_{k,p}(\nu_n),C_{k,p}(\mu)) \ge \varepsilon.
	\end{equation*}
	Now use Proposition~\ref{prop:clustering-cty} with $k_n = k$ and $R_n = \mathcal{X}$ for all $n\in\N$, which guarantees that we have $\vec{D}_{\textnormal{H}}(C_{k,p}(\nu_n),C_{k,p}(\nu))\to 0$, hence $\vec{D}_{\textnormal{H}}(C_{k,p}(\nu),C_{k,p}(\mu)) \ge \varepsilon$, by the above.
	In particular, $\nu\in A$, thus $A$ is $\weak^p$-closed.
	
	Next, note that $\mu$ is a Borel probability measure on a Polish metric space with all exponential moments finite, so we conclude via \cite[Theorem~1.1]{WWW} that $\{\bar \mu_n\}_{n\in\N}$ satisfy a large deviations principle in $(\mathcal{P}_p(\mathcal{X}),\weak^p)$ with good rate function $H(\cdot|\mu):\mathcal{P}_p(\mathcal{X})\to  [0,\infty]$.
	In particular, the large deviations upper bound implies
	\begin{equation*}
	\begin{split}
	\limsup_{n\to\infty}\frac{1}{n}\log \P\left(\vec{D}_{\textnormal{H}}(C_{k,p}(\bar \mu_n),C_{k,p}(\mu)) \ge \varepsilon\right) &= \limsup_{n\to\infty}\frac{1}{n}\log \P(\bar \mu_n \in A) \\
	&\le -\inf\{H(\nu|\mu): \nu\in A \} := c_{k,p}(\mu,\varepsilon).
	\end{split}
	\end{equation*}
	Finally, assume towards a contradiction that $c_{k,p}(\mu,\varepsilon)=0$, so that there exist $\{\nu_n\}_{n\in\N}$ in $A$ with $H(\nu_n|\mu)\to 0$.
	Then Lemma~A.2 in the supplementary material implies $\nu_n\to \mu$ in $\weak^p$ so Proposition~\ref{prop:clustering-cty} implies $\vec{D}_{\textnormal{H}}(C_{k,p}(\nu_n),C_{k,p}(\mu))\to 0$.
	This is impossible since $\nu_n\in A$ for all $n\in\N$, hence we must have $c_{k,p}(\mu,\varepsilon)>0$.
\end{proof}

Now, we make some remarks on possible limitations and extensions of this result.
First of all, the constant $c_{k,p}(\mu,\varepsilon)$ appearing as the exponential rate of decay has an exact characterization as
\begin{equation}
c_{k,p}(\mu,\varepsilon) := \inf\{H(\nu|\mu): \nu\in \mathcal{P}_p(\mathcal{X}),  \dvecHaus(C_{k,p}(\nu),C_{k,p}(\mu)) \ge \varepsilon \}.
\end{equation}
From this form we have shown $c_{k,p}(\mu,\varepsilon)> 0$ for all $\varepsilon > 0$, but it appears difficult to say much else.
We believe it would be interesting to try to find some simple geometries $(\mathcal{X},d)$ and simple distributions $\mu\in \mathcal{P}(\mathcal{X})$ for which $c_{k,p}(\mu,\varepsilon)$ can be exactly or approximately computed.
For example, if $\mu$ is compactly-supported, then do we have $c_{k,p}(\mu,\varepsilon)\lesssim\varepsilon^{-2}$, which can be interpreted as a sort of asymptotically sub-Gaussian concentration?
For $k=1$, this is guaranteed by Hoeffding's lemma in Euclidean spaces and by \cite{JaffeSantoroLDP} in Riemannian manifolds and the Wasserstein space.

Second, we address the $L^p$-exponential moment assumption.
As we show below, for $\mathcal{X}=\R^m$ and $p=1$, the assumption holds if the moment generating function of $\mu$ is finite everywhere. 
However, for $p=2$, the condition is too strict to even cover the case where $\mu$ is a standard Gaussian distribution; we conjecture that the condition is not necessary for the conclusion to hold.

\begin{lemma}
	If $\mu\in \mathcal{P}(\R^m)$ has $\int_{\R^m}\exp(\langle \lambda,y\rangle)\, \diff\mu(y)  < \infty$ for all $\lambda \in\R^m$, then we have $\int_{\R^m}\exp(\alpha |x-y|)\, \diff\mu(y) < \infty$ for all $x\in\R^m$ and $\alpha > 0$.
\end{lemma}

\begin{proof}
	Let $N$ be a $\frac{1}{2}$-net of the unit ball $B :=\{\lambda\in \R^m: |\lambda|\le 1 \}$.
	Then for any $\alpha >0$ and $y\in \R^m$, use the standard duality of the $L^2$ inner product and Cauchy-Schwarz to get:
	\begin{equation*}
	\begin{split}
	\exp(2\alpha |y|) =\max_{\lambda\in B} \exp(2\alpha \langle \lambda,y\rangle) &=\max_{\lambda\in B} \min_{\lambda'\in N}\exp(2\alpha \langle \lambda',y\rangle)\exp(2\alpha \langle \lambda-\lambda',y\rangle) \\
	&\le  \exp(\alpha |y|)\max_{\lambda'\in N} \exp(2\alpha \langle \lambda',y\rangle).
	\end{split}
	\end{equation*}
	Rearranging this yields
	\begin{equation*}
	\exp(\alpha |y|) \le \max_{\lambda\in N} \exp(2\alpha \langle \lambda,y\rangle) \le \sum_{\lambda\in N} \exp(2\alpha \langle \lambda,y\rangle).
	\end{equation*}
	Thus, by integrating the above, we have, for any $x\in \R^m$ and $\alpha >0$:
	\begin{equation*}
	\int_{\R^m}\exp(\alpha|x-y|)\, \diff\mu(y) \le \exp(\alpha|x|)\sum_{\lambda\in N} \int_{\R^m}\exp(2\alpha \langle \lambda,y\rangle)\, \diff\mu(y).
	\end{equation*}
	Since the right side is finite by assumption, this proves the claim.
\end{proof}


%
%

\begin{acks}[Acknowledgments]
 The authors thanks the anonymous referees for comments that signifcantly improved the scope and quality of the paper.
 The author also thanks Mo\"ise Blanchard for many useful conversations.
\end{acks}


\begin{funding}
This material is based upon work for which the author was supported by the National Science Foundation Graduate Research Fellowship under Grant No. DGE 1752814
\end{funding}



\bibliographystyle{imsart-number} 
\bibliography{clustering}       


\appendix


	\section{Auxiliary Analytic Results}
	
	In this appendix we prove some auxiliary analysis results that are used in the main body of the paper.
	We believe these results are well-known, but we could not find appropriate references so we prove them for the sake of completeness.
	
	\begin{lemma}
		Suppose $(\mathcal{X},d)$ is a separable metric space and $\mu\in\mathcal{P}(\mathcal{X})$.
		If $\{B_{\gamma}\}_{\gamma\in\Gamma}$ is an arbitrary collection of $d$-closed sets with $\mu(B_{\gamma}) = 1$ for all $\gamma\in\Gamma$, then the $d$-closed set $B:=\bigcap_{\gamma\in\Gamma}B_{\gamma}$ satisfies $\mu(B)=1$.
	\end{lemma}
	
	\begin{proof}
		By taking complements, the result is equivalent to the following:
		If $\{U_{\gamma}\}_{\gamma\in\Gamma}$ is an arbitrary collection of $d$-open sets with $\mu(U_{\gamma}) = 0$ for all $\gamma\in\Gamma$, then the $d$-open set $U:=\bigcup_{\gamma\in\Gamma}U_{\gamma}$ satisfies $\mu(U)=1$.
		To prove this, we use the fact that $(\mathcal{X},d)$ is a separable metric space to get a countable collection $\{U_n\}_{n\in\N}$ of $d$-open sets which form a basis for the topology generated by $d$.
		(For example, let $\{U_n\}_{n\in\N}$ consist of all $d$-open balls with rational radii and whose centers form a countable dense subset of $\mathcal{X}$.)
		For each $\gamma\in\Gamma$ we can get an index set $N(\gamma)\subseteq \N$ such that $U_{\gamma}=\bigcup_{n\in N(\Gamma)}U_n$, and it follows that we have $U=\bigcup_{n\in N}U_n$, where $N:=\bigcup_{\gamma\in\Gamma}N(\gamma)$ is a countable index set.
		Since
		\begin{equation*}
			\mu(U) = \mu\left(\bigcup_{n\in N}U_n\right) \le \sum_{n\in N}\mu(U_n),
		\end{equation*}
		it suffices to prove $\mu(U_n) = 0$ for every $n\in N$.
		Indeed, for any $n\in N$ there exists $\gamma\in\Gamma$ with $n\in N(\gamma)$ with $U_n\subseteq U_{\gamma}$ hence $\mu(U_n) \le \mu(U_{\gamma}) = 0$.
	\end{proof}
	
	\begin{lemma}\label{lem:KL-exp-mom-weakp}
		Suppose $(\mathcal{X},d)$ is a Polish metric space and $\{\mu_n\}_{n\in\N}$ and $\mu$ in $\mathcal{P}_{p}(\mathcal{X})$ are such that we have $H(\mu_n | \mu) \to 0$ and $\int_{\mathcal{X}}\exp(\alpha d^p(x,y))\diff\mu(y) < \infty$ for some $x\in \mathcal{X}$ and $\alpha > 0$.
		Then, $\mu_n\to \mu$ in $\weak^{p}$.
	\end{lemma}
	
	\begin{proof}
		First, note by Pinsker's inequality \cite[Lemma~2.5]{Tsybakov} that $H(\mu_n|\mu) \to 0$ implies $\mu_n\to \mu$ in $\weak$.
		Moreover, by the Portmanteau lemma, since $d^p(x,\cdot):X\to \R$ is non-negative and continuous, we have
		\begin{equation}\label{eqn:wp-2}
			\int_{\mathcal{X}}d^p(x,y)\diff\mu(y)\le \liminf_{n\to\infty}\int_{\mathcal{X}} d^p(x,y)\diff\mu_n(y).
		\end{equation}
		Next recall that the Donsker-Varadhan variational principle  \cite[Lemma~6.2.13]{DZ} states
		\begin{equation}\label{eqn:Donsker-Varadhan}
			H(\mu_n | \mu) = \sup_{\phi\in b\mathcal{B}(\mathcal{X})}\left(\int_{\mathcal{X}} \phi\diff\mu_n - \log\left(\int_{\mathcal{X}}\exp(\phi)\diff\mu\right)\right),
		\end{equation}
		where $b\mathcal{B}(\mathcal{X})$ represents the space of all bounded measurable functions from $\mathcal{X}$ to $\R$.
		Now fix $\beta \in (0,\alpha)$.
		For $R > 0$ arbitrary, set $\phi_R:\mathcal{X}\to\R$ via $\phi_R(y) := \beta d^{p}(x,y)\ind_{\bar B^d_R(x)}(y)$ for $y\in \mathcal{X}$, and note that $\phi_R\in b\mathcal{B}(\mathcal{X})$.
		Moreover, we have $\phi_R(y)\uparrow \beta d^p(x,y)$ as $R\to\infty$ for all $y\in \mathcal{X}$.
		By applying \eqref{eqn:Donsker-Varadhan} for each $R > 0$  we have
		\begin{equation*}
			H(\mu_n | \mu) \ge \int_{\mathcal{X}} \phi_R\diff\mu_n - \log\left(\int_{\mathcal{X}}\exp(\phi_R)\diff\mu\right),
		\end{equation*}
		so taking $R\to \infty$ and applying monotone convergence gives
		\begin{equation*}
			H(\mu_n | \mu) \ge \int_{\mathcal{X}}\beta d^p(x,y)\diff\mu_n(y) - \log\left(\int_{\mathcal{X}}\exp(\beta d^p(x,y))\diff\mu(y)\right).
		\end{equation*}
		Now we send $n\to\infty$ and rearrange to get
		\begin{equation*}
			\limsup_{n\to\infty}\int_{\mathcal{X}}d^p(x,y)\diff\mu_n(y) \le \frac{1}{\beta}\log\left(\int_{\mathcal{X}}\exp(\beta d^p(x,y))\diff\mu(y)\right).
		\end{equation*}
		The last step is to take the limit on the right side as $\beta\to 0$.
		To do this, we define the function $f:[0,\alpha]\to\R$ via $f(\beta) := \int_{\mathcal{X}}\exp(\beta d^p(x,y))\diff\mu(y)$, which obviously satisfies $f(0) = 1$, and we note that the desired limit is exactly $f'(0)$.
		In fact, we can see from dominated convergence that $f$ is differentiable at 0 and that we can interchange differentiation and integration, since we have
		\begin{equation*}
			\frac{f(\beta)}{\beta} = \int_{\mathcal{X}}\frac{\exp(\beta d^p(x,y))-1}{\beta}\diff\mu(y),
		\end{equation*}
		and since the integrand is upper bounded, uniformly over $\beta\in [0,\alpha]$, by
		\begin{equation*}
			\frac{\exp(\beta d^p(x,y))-1}{\beta} \le \exp(\beta d^p(x,y))  \le \exp(\alpha d^p(x,y)),
		\end{equation*}
		which is $\mu$-integrable by assumption.
		Therefore,
		\begin{equation}\label{eqn:wp-1}
			\limsup_{n\to\infty}\int_{\mathcal{X}}d^p(x,y)\diff\mu_n(y) \le f'(0) = \int_{\mathcal{X}}d^p(x,y)\diff\mu(y).
		\end{equation}
		Finally, combining \eqref{eqn:wp-1} and \eqref{eqn:wp-2} yields $\mu_n\to \mu$ in $\weak^p$.
	\end{proof}
	
	\section{Proof of Proposition~\ref{prop:bdd}}
	
	Let $L$ denote the set of all $\ell\in \{0,1,\ldots, k\}$ for which there exists $z_{\ell}\in \mathcal{X}$ and $r_{\ell} > 0$ such that for all sufficiently large $n\in\N$ and all $S_n\in C_p(\mu_n,k_n,R_n;\varepsilon_n)$, we have $\#(S_n\cap B_{r_{\ell}}(x_{\ell})) \ge \ell$.
	We claim that $L = \{0,1,\ldots, k\}$; if this holds, then the result follows by taking $z := z_{k}$ and $r := r_{k}$.
	Since clearly $0\in L$, we see that $L$ is non-empty.
	Now, we assume for the sake of contradiction that $L\neq \{0,1,\ldots, k\}$.
	
	Under this assumption, we can let $\ell$ denote the largest element of $L$.
	Thus, $\ell\in L$ and $\ell+1\neq L$.
	Now use $\ell\in L$ to get some $z_{\ell}\in \mathcal{X}$ and $r_{\ell} > 0$ such that for all sufficiently large $n\in\N$ and all $S_n\in C_p(\mu_n,k_n,R_n;\varepsilon_n)$, we have,
	\begin{equation*}
		\#(S_n\cap B_{r_{\ell}}(z_{\ell})) \ge \ell.
	\end{equation*}
	Next, use $\ell+1\notin L$ inductively to get a sequence $\{n_j\}_{j\in\N}$ with the following property:
	For each $j\in\N$ there exists $S_j\in C_p(\mu_{n_j},k_{n_j},R_{n_j};\varepsilon_{n_j})$ with
	\begin{equation*}
		\#(S_j\cap B_{r_{\ell}+j}(z_{\ell})) < \ell+1.
	\end{equation*}
	On the other hand, we also have
	\begin{equation*}
		\#(S_j\cap B_{r_{\ell}+j}(z_{\ell})) \ge \#(S_n\cap B_{r_{\ell}}(z_{\ell})) \ge \ell.
	\end{equation*}
	Since these finite cardinalities are all integers, this means we have
	\begin{equation*}
		\#(S_j\cap B_{r_{\ell}+j}(z_{\ell})) = \ell
	\end{equation*}
	for all $j\in\N$.
	In particular, a point $x\in S_j\setminus B_{r_{\ell}+j}(z_{\ell})$ for $j\in\N$ must have $d(x,z_{\ell}) \ge j+r_{\ell}$.
	
	Next, we define $T_j := S_j \cap B_{r_{\ell}}(z_{\ell})$ for all $j\in\N$, and we investigate the asymptotics of the excess loss $W_p(T_j,\mu_{n_j}) - W_p(S_j,\mu_{n_j})$ as $j\to\infty$.
	To do this, we get a probability space $(\Omega,\F,\P)$, with expectation $\E$, and random variables $\{Y^j\}_{j\in\N}$ and $Y$ as in Lemma~\ref{lem:Skor}.
	Note that by construction we have $S_j\setminus T_j \subseteq \mathcal{X}\setminus B_{r_{\ell}+j}(z_{\ell})$, hence
	\begin{equation*}
		|d(Y,T_j) - d(Y,S_j)|\to 0
	\end{equation*}
	almost surely.
	Also we have
	\begin{equation*}
		|d(Y^{n_j},S_j) - d(Y,S_j)| \le d(Y^{n_j},Y) \to 0
	\end{equation*}
	and
	\begin{equation*}
		|d(Y^{n_j},T_j) - d(Y,T_j)| \le d(Y^{n_j},Y) \to 0
	\end{equation*}
	both almost surely.
	Combining the three preceding displays with the triangle inequality gives
	\begin{equation}\label{eqn:main-1}
		|d(Y^{n_j},T_j)-d(Y^{n_j},S_j)|\to 0
	\end{equation}
	almost surely.
	Now recall the elementary fact that if $\{a_n\}_{n\in\N}$  and $\{b_n\}_{n\in\N}$  are sequences with $|a_n-b_n|\to 0$ such that $\{b_n\}_{n\in\N}$ is bounded, then we have $|a_n^p-b_n^p|\to 0$ for any fixed $p\ge 1$.
	Thus, combining \eqref{eqn:main-1} with the fact that $\{d(Y^{n_j},T_j)\}_{j\in\N}$ is bounded, we conclude
	\begin{equation}\label{eqn:main-2}
		|d^p(Y^{n_j},T_j)-d^p(Y^{n_j},S_j)|\to 0
	\end{equation}
	almost surely.
	Also note by equation \eqref{eqn:dist-comparison} that we have
	\begin{align*}
		|d^p(Y^{n_j},T_j)-d^p(Y^{n_j},S_j)| &\le d^p(Y^{n_j},T_j) \\
		&\le 2^{p-1}(d^p(Y^{n_j},z_{\ell}) + r_{\ell}^p).
	\end{align*}
	Since $\mu_{n_j}\to \mu$ in $\weakp$ implies that $\{d^p(Y^{n_j},z_{\ell})\}_{n\in\N}$ is uniformly integrable, we conclude that $\{d^p(Y^{n_j},T_j)-d^p(Y^{n_j},S_j)\}_{n\in\N}$ is uniformly integrable.
	Combining this with \eqref{eqn:main-2} yields
	\begin{equation*}
		\E[d^p(Y^{n_j},T_j)-d^p(Y^{n_j},S_j)]\to 0
	\end{equation*}
	which rearranges to
	\begin{equation}\label{eqn:main-4}
		\limsup_{j\to\infty}W_p(T_j,\mu_{n_j}) \le \liminf_{j\to\infty}W_p(S_j,\mu_{n_j})
	\end{equation}
	as desired.
	
	Finally, let us show that this analysis gives a contradiction.
	To do this, take an arbitrary $S'\subseteq R$ with $1\le \#S'\le k$.
	By Lemma~\ref{lem:Kur-lower-subseq} we can get a subsequence $\{j_i\}_{i\in\N}$ and a set $S_i'\subseteq R_{n_{j_{i}}}$ with $\#S_i' = \#S_{j_{i}}$ for each $i\in\N$, such that we have $S'_i\to S'$ in $\dHaus$.
	Then Lemma~4 in the main body implies
	\begin{equation*}
		W_p(S_{j_{i}},\mu_{n_{j_{i}}}) \le W_p(S_i',\mu_{n_{j_{i}}}) + \varepsilon_{n_{j_{i}}} \to W_p(S',\mu)
	\end{equation*}
	as $i\to\infty$.
	Combining this with \eqref{eqn:main-4} and applying Lemma~4 and Lemma~5 in the main body, we get
	\begin{align*}
		\limsup_{j\to\infty}W_p(T_{j},\mu) &= \limsup_{j\to\infty}W_p(T_{j},\mu_{n_{j}}) \\
		&\le \liminf_{j\to\infty}W_p(S_j,\mu_{n_j}) \\
		&\le \liminf_{i\to\infty}W_p(S_{j_i},\mu_{n_{j_i}}) \\
		&\le \liminf_{i\to\infty}(W_p(S_i',\mu_{n_{j_{i}}}) + \varepsilon_{n_{j_{i}}}) \\
		&= W_p(S',\mu).
	\end{align*}
	Taking the infimum over all feasible $S'$ and using the non-singularity of $(\mu,k,R)$, we have
	\begin{equation*}
		\limsup_{i\to\infty}W_p(T_{j_{i}},\mu_{n_{j_{i}}}) \le m_{k,p}(\mu,R) < m_{\ell,p}(\mu,R).
	\end{equation*}
	But this is a contradiction since we have $\#T_{j_{i}} = \ell$ for all $i\in\N$ by construction.
	Hence, we must have $L = \{0,1,\ldots, k\}$, and the result is proved.
	
	\section{Proof of Proposition~\ref{prop:elbow-k-cts}}
	
	First, we claim that we have $m_{k,p}(\mu) \to 0$ as $k\to\infty$.
	To do this, use Prokhorov's theorem to get, for each $\eta>0$, a compact $L_{\eta}\subseteq \mathcal{X}$ such that $\mu(\mathcal{X}\setminus L_{\eta})\le \eta$.
	In particular, we have $\ind\{y\notin L_{\eta}\}\to 0$ as $\eta\to0$ holding $\mu$-almost surely for all $y\in \mathcal{X}$.
	Now fix $o\in \mathcal{X}$, let $N_{\eta}^0$ be an $\eta^{1/p}$-net of $L_{\eta}$, and set $N_{\eta} := N_{\eta}^{0}\cup\{o\}$.
	It follows that we have
	\begin{align*}
		\int_{\mathcal{X}}\min_{x\in N_{\eta}}d^p(x,y)\diff\mu(y) &= \int_{L_{\eta}}\min_{x\in N_{\eta}}d^p(x,y)\diff\mu(y) + \int_{X\setminus L_{\eta}}\min_{x\in N_{\eta}}d^p(x,y)\diff\mu(y) \\
		&\le \eta\cdot\mu(L_{\eta}) + \int_{\mathcal{X}\setminus L_{\eta}}d^p(o,y)\diff\mu(y) \to 0
	\end{align*}
	as $\eta\to 0$, where we applied dominated convergence to the last term.
	Writing $\ell_{\eta} := \#N_{\eta}$ for $\eta>0$, this implies $m_{\ell_{\eta},p}(\mu) \to 0$ as $\eta\to 0$.
	Since $\{\ell_{\eta}\}_{\eta>0}$ is non-decreasing as $\eta\to 0$, this further implies $m_{k,p}(\mu)\to 0$ as $k\to\infty$.
	
	Now we consider the value $M := \max\{\Delta^2m_{k,p}(\mu): k\in\N, k \ge 2\}$.
	It is clear that $m_{k,p}(\mu)\to 0$ implies $\Delta^2m_{k,p}(\mu)\to 0$ hence $M<\infty$.
	Next let us also show $M>0$.
	To do this, assume for the sake of contradiction that $M\le 0$, so $\Delta^2m_{k,p}(\mu) \le 0$ for all $k\ge 2$.
	Then for all $2\le j \le j'$, we have
	\begin{equation*}
		m_{j,p}(\mu) - m_{j+1,p}(\mu) - (m_{j',p}(\mu)-m_{j'+1,p}(\mu)) =\sum_{k=j}^{j'}\Delta^2m_{k,p}(\mu) \le 0.
	\end{equation*}
	Sending $j'\to\infty$, we conclude $m_{j,p}(\mu) \le m_{j+1,p}(\mu)$ for all $j\ge 2$.
	In other words $j\mapsto m_{j,p}(\mu)$ is non-decreasing for $j\ge 2$.
	But it is obviously non-increasing by definition, so it must in fact be constant.
	Now note that the non-singularity of $(\mu,3,\mathcal{X})$ implies $m_{2,p}(\mu) > 0$, hence $\liminf_{k\to\infty}m_{k,p}(\mu) > 0$.
	This contradicts the conclusion of the preceding paragraph, so we must have $0<M<\infty$.
	Furthermore, by assumption there is a unique $k_{\ast} :=\kelbow_p(\mu) \in\N$ such that $M = \Delta^2m_{k_{\ast},p}(\mu)$.
	
	Now suppose that $\{\mu_n\}_{n\in\N}$ in $\mathcal{P}_p(\mathcal{X})$ have $\mu_n\to\mu$ in $\weak^p$.
	The result is proved if we can show that for each subsequence $\{n_j\}_{j\in\N}$ there exists a further subsequence $\{j_i\}_{i\in\N}$ satisfying $\kelbow_p(\mu_{n_{j_i}})\to k_{\ast}$ as $i\to\infty$.
	To do this, we use the construction from above to choose $\eta>0$ sufficiently small so that
	\begin{equation*}
		\int_{\mathcal{X}\setminus L_{\eta}}d^p(o,y)\diff\mu(y) \le \frac{M}{16}.
	\end{equation*}
	Then, we set $K:=L_{\eta}$, we let $N^0$ denote a $\frac{1}{2}(\frac{M}{8})^{1/p}$-net of $K$, and we define $N:=N^0\cup\{o\}$ and $\ell:=\#N$.
	Also use Lemma~\ref{lem:Skor} to construct a probability space $(\Omega,\F,\P)$, with expectation denoted $\E$, on which are defined random variables $\{Y^n\}_{n\in\N}$ and $Y$ with laws $\{\mu_{n}\}_{n\in\N}$ and $\mu$, respectively, such that we have $\E[d^p(Y^n,Y)]\to 0$ as $n\to\infty$.
	Now observe that for any subsequence $\{n_j\}_{j\in\N}$ we can choose a further subsequence $\{j_i\}_{i\in\N}$ satisfying both
	\begin{equation}\label{eqn:elbows-1}
		\limsup_{j\to\infty}\sup_{k\ge \ell}m_{k,p}(\mu_{n_j}) = \lim_{i\to\infty}\sup_{k\ge \ell}m_{k,p}(\mu_{n_{j_i}})
	\end{equation}
	and
	\begin{equation}\label{eqn:elbows-2}
		\E[d^p(Y^{n_{j_i}},Y)]\le 2^{-i} \qquad\textrm{ for all } \qquad i\in\N.
	\end{equation}
	Now write $a_i := i^22^{-i}$, and observe by Markov's inequality that we have
	\begin{equation*}
		\sum_{i\in\N}\P\left(d^p(Y^i,Y)\ge a_i\right) \le \sum_{i\in\N}\frac{\E\left[d^p(Y^i,Y)\right]}{a_i} \le \sum_{i\in\N}\frac{1}{i^2} < \infty.
	\end{equation*}
	Thus, by Borel-Cantelli, we have $\P(d^p(Y^i,Y)\le a_i \text{ for sufficiently large } i\in\N) = 1$.
	
	Moreover, if we define $K_i:= \{y\in \mathcal{X}: d^p(y,K)\le a_i\}$ for each $i\in\N$, then we have
	\begin{equation*}
		\limsup_{i\to\infty}\ind\{Y^i\notin K_i\} \le \ind\{Y\notin K\}
	\end{equation*}
	almost surely.
	Now we can apply Fatou's lemma to get
	\begin{align*}
		\limsup_{i\to\infty}\int_{\mathcal{X}\setminus K_i}d^p(o,y)\diff\mu_{n_{j_i}}(y) &= \limsup_{i\to\infty}\E\left[d^p(o,Y^i)\ind\{Y^i\notin K_i\}\right] \\
		&\le \E\left[\limsup_{i\to\infty}d^p(o,Y^i)\ind\{Y^i\notin K_i\}\right] \\
		&\le \E\left[d^p(o,Y)\ind\{Y\notin K\}\right] \\
		&= \int_{\mathcal{X}\setminus K}d^p(o,y)\diff\mu(y).
	\end{align*}
	By equation \eqref{eqn:dist-comparison} in the main body, we also have $d^p(y,N) \le \frac{M}{16}+2^{p-1}a_i$ for all $y\in K_i$.
	Therefore,
	\begin{equation*}
		\begin{split}
			\sup_{k\ge \ell}m_{k,p}(\mu_{n_{j_i}}) &= \sup_{k\ge \ell} \inf_{\substack{S\subseteq X \\ 1\le\#S\le k}}\int_{\mathcal{X}}\min_{x\in S}d^p(x,y)\diff\mu_{n_{j_i}}(y) \\
			&\le \int_{\mathcal{X}}\min_{x\in N}d^p(x,y)\diff\mu_{n_{j_i}}(y) \\
			&= \int_{K_i}\min_{x\in N}d^p(x,y)\diff\mu_{n_{j_i}}(y) + \int_{\mathcal{X}\setminus K_i}\min_{x\in N}d^p(x,y)\diff\mu_{n_{j_i}}(y) \\
			&\le \mu_{n_{j_i}}(K_i)\left(\frac{M}{16} +2^{p-1}a_i\right) + \int_{\mathcal{X}\setminus K_i}d^p(o,y)\diff\mu_{n_{j_i}}(y).
		\end{split}
	\end{equation*}
	Thus, taking $i\to\infty$ yields
	\begin{equation*}
		\limsup_{j\to\infty}\sup_{k\ge \ell}m_{k,p}(\mu_{n_{j}}) = \lim_{i\to\infty}\sup_{k\ge \ell}m_{k,p}(\mu_{n_{j_i}}) \le \frac{M}{8}.
	\end{equation*}
	Now simply apply the triangle inequality to get
	\begin{equation*}
		\limsup_{j\to\infty}\sup_{k\ge \ell}\Delta^2m_{k,p}(\mu_{n_j}) \le \frac{M}{2}.
	\end{equation*}
	This means that, for sufficiently large $j\in\N$, the maximum of $\{\Delta^2m_{k,p}(\mu_{n_j}): k\in\N\}$ is not achieved on $\{\ell,\ell+1,\ldots\}$.
	
	Finally, set
	\begin{equation*}
		\varepsilon := \min\left\{|\Delta^2m_{k_{\ast},p}(\mu)-\Delta^2m_{k,p}(\mu)|: k\le \ell, k\neq k_{\ast}\right\} >0.
	\end{equation*}
	By Lemma~\ref{lem:nonsingular-limits} in the main body, there is sufficiently large $j\in\N$ such that we have $|m_{k,p}(\mu_{n_j})-m_{k,p}(\mu)| < \frac{\varepsilon}{8}$ for all $k\le \ell$.
	Combining this with the above and the triangle inequality shows that, for sufficiently large $j\in\N$, we have
	\begin{equation*}
		\arg\max\{\Delta^2m_{k,p}(\mu_{n_j}): k\in\N \} = \arg\max\{\Delta^2m_{k,p}(\mu): k\in\N \} = \{k_{\ast}\},
	\end{equation*}
	hence $\kelbow_p(\mu_{n_j}) = k_{\ast}$.
	

\end{document}